\documentclass{amsart}   	
\usepackage{geometry}                		
\usepackage{graphicx}				
\usepackage{amsmath}								
\usepackage{amssymb}
\usepackage{amsthm}
\usepackage{fancyhdr}

\newtheorem{theorem}{Theorem}[section]

\newtheorem{lemma}{Lemma}[section]
\newtheorem{prop}{Proposition}[section]
\newtheorem{remark}{Remark}[theorem]
\newtheorem{cor}{Corollary}[theorem]

\title{Remarks on symplectic mean curvature flows in K\"ahler surfaces with positive holomorphic sectional curvatures}
\author{Shijin Zhang}
\address{School of Mathematics and Systems Science, Beihang University, Beijing 100191, P.R. China}
\email{shijinzhang@buaa.edu.cn}
\date{}							

\begin{document}
\pagestyle{fancy}
\fancyhead{}
\fancyhead[CO]{Symplectic mean curvature flows}
\fancyhead[CE]{Shijin Zhang}
\maketitle
\begin{abstract}
\noindent
In this paper, we  mainly study the mean curvature flow in K\"ahler surfaces with positive holomorphic sectional curvatures. First, we prove that if the ratio $\lambda$ of the maximum and the minimum of the holomorphic sectional curvatures $< 2$, then there exists a positive constant $\delta>\frac{29(\lambda-1)}{\sqrt{(48-24\lambda)^{2}+(29\lambda-29)^{2}}}$ such that $\cos\alpha\geq\delta$ is preserved along the flow, improving the main theorem in [LY]; Secondly, as similar as the main theorem in [HL0], we prove that when $\cos\alpha$ is close to $1$ enough, then the symplectic mean curvature flow exists for long time and converges to a holomorphic curve; Finally, we prove that the symplectic mean curvature flow on K\"ahler surfaces with $\lambda\leq 1+\frac{1}{200}$ exists for long time and converges to a holomorphic curve if the initial surface satisfies a pinching condition, which generalize one of the main theorems in [HLY].
\end{abstract}

\section{Introduction}
Let $(M, J, \omega, \overline{g})$ be a K\"ahler surface. For a compact oriented real surface $\Sigma$ which is smoothly immersed in $M$, the K\"ahler angle $\alpha$ of $\Sigma$ in $M$ was defined by
$$\omega|_{\Sigma}=\cos\alpha d\mu_{\Sigma}$$
where $d\mu_{\Sigma}$ is the area element of $\Sigma$ in the induced metric from $g$. We say that $\Sigma$ is a symplectic surface if $\cos\alpha>0$; $\Sigma$ is a holomorphic curve if $\cos\alpha=1$.

It is important to find the conditions to assure that the symplectic property is preserved along the mean curvature flow. In the case that $M$ is a K\"ahler-Einstein surface, the symplectic property is preserved. If the ambient K\"abler surface evolves along the K\"ahler-Ricci flow, Han and Li [HL1] proved that the symplectic property is also preserved. In [LY], Li and Yang found another condition to assure that the symplectic property is preserved along the mean curvature flow. In this note, we will improving their conditions to assure that along the flow.

In this paper we only consider the ambient K\"ahler surface with positive holomorphic sectional curvature. Denote the minimum and maximum of holomorphic sectional curvatures of $M$ by $k_{1}$ and $k_{2}$, and $\lambda=\frac{k_{2}}{k_{1}}$. Then we have the first theorem.
\begin{theorem}\label{Thm1}
Suppose $M$ is a K\"abler surface with positive holomorphic sectional curvatures. If $1\leq \lambda <2$ and $\cos\alpha(\cdot,0)\geq \delta >\frac{29(\lambda-1)}{\sqrt{(48-24\lambda)^{2}+(29\lambda-29)^{2}}}$, then
along the flow
\begin{equation}
(\frac{\partial}{\partial t}-\Delta)\cos\alpha \geq |\overline{\nabla}J_{\Sigma_{t}}|^{2}\cos\alpha +C\sin^{2}\alpha,
\end{equation}
where $C$ is a positive constant depending only on $k_1,k_2$ and $\delta$. As a corollary, $\min_{\Sigma_{t}}\cos\alpha$ is increasing with respect to $t$. In particular, at each time $t$, $\Sigma_{t}$ is symplectic.
\end{theorem}
\begin{remark}
The main theorem in [LY], the lower bound of $\delta$ is $\frac{53(\lambda-1)}{\sqrt{(53\lambda-53)^{2}+(48-24\lambda)^{2}}}$ for $\lambda \in [1, \frac{11}{7})$, and $\frac{8\lambda-5}{\sqrt{(8\lambda-5)^{2}+(12-6\lambda)^{2}}}$ for $\lambda \in [\frac{11}{7}, 2)$. It is easy to check that for each $\lambda \in [1, 2)$, then
$$\frac{29(\lambda-1)}{\sqrt{(48-24\lambda)^{2}+(29\lambda-29)^{2}}}\leq \min\{\frac{53(\lambda-1)}{\sqrt{(53\lambda-53)^{2}+(48-24\lambda)^{2}}}, \frac{8\lambda-5}{\sqrt{(8\lambda-5)^{2}+(12-6\lambda)^{2}}}\}.$$
Hence we improving the result of the Main Theorem in [LY].
\end{remark}

In analogy  to the main theorem in [HL0], we also prove the following theorem for a K\"ahler surface with positive holomorphic sectional curvature and $1\leq\lambda <2$.
 \begin{theorem}
Suppose that $M$ is a K\"ahler surface with positive holomorphic sectional curvature and $1\leq \lambda <2$. Let $\alpha$ be the K\"ahler angle of the surface $\Sigma_{t}$ which evolves by the mean curvature flow. Suppose that $\cos\alpha(\cdot,0)>\frac{58(\lambda-1)}{\sqrt{(48-24)^{2}+(58\lambda-58)^{2}}}$.Then there exists a sufficiently small constant $\epsilon_{1}$ such that if $\int_{\Sigma_{0}}\frac{\sin^{2}\alpha(x,0)}{\cos\alpha(x,0)}d\mu_{0}\leq \epsilon_{1}$ and $\epsilon_{1}$ satisfying
\[\epsilon_{1}\leq \frac{\pi^{2}\epsilon_{0}^{2}r_{0}^{6}(1-e^{-\frac{3}{8}(2-\lambda)k_{1}})^{2}}{4 Area(\Sigma_{0})},\]
there $r_{0}$ is defined in Remark 4.1.1, and $\epsilon_{0}$ is the constant in White's theorem (Theorem 4.1 in the present paper), the mean curvature flow with initial surface $\Sigma_{0}$ exists globally and it converges to a holomorphic curve.
\end{theorem}

As a consequence, we obtain that
\begin{cor}
Suppose that $M$ is a K\"ahler surface with positive holomorphic sectional curvature and $1\leq \lambda <2$. Let $\alpha$ be the K\"ahler angle of the surface $\Sigma_{t}$ which evolves by the mean curvature flow. If there exists a positive constant $\delta$ such that
\[\cos\alpha(\cdot,0)\geq \delta > \frac{29(\lambda-1)}{\sqrt{(48-24\lambda)^{2}+(29\lambda-29)^{2}}}\]
and has the longtime existence. Then the mean curvature flow converges to a holomorphic curve.
\end{cor}
Han, Li and Yang proved the following theorem with constant holomorphic sectional curvature, see Theorem 3.2, Theorem 4.1 and Theorem 4.2 in [HLY].
\begin{theorem}[Han-Li-Yang]
Suppose $\Sigma$ is a symplectic surface in $CP^{2}$ with constant holomorphic sectional curvature $k>0$. Assume that $|A|^{2}\leq \sigma |H|^{2}+\frac{2\sigma-1}{\sigma}k$ and $\cos\alpha\geq \sqrt{\frac{7\sigma-3}{3\sigma}}$ holds on the initial surface for any $\frac{1}{2}<\sigma\leq\frac{2}{3}$, then it remains true along the symplectic mean curvature flow. Furthermore, the symplectic mean curvature flow exists for long time and converges to a holomorphic curve at infinity.
\end{theorem}

In this paper, we also prove the similar theorem with positive holomorphic sectional curvature.
\begin{theorem}\label{Thm1.3}
Suppose $\Sigma$ is a symplectic surface in the K\"ahler surface $(M,J,\omega,\overline{g})$ with positive holomorphic sectional curvature and $1\leq \lambda <1+\frac{1}{200}$, and $|\overline{\nabla}Rm|\leq Kk_{1}(\lambda-1)$ for a positive constant $K\leq\min\{2, 2k_{1}\}$. For any $\sigma$ satisfying $\frac{1}{2}+\frac{24(\lambda-1)}{1-34(\lambda-1)}<\sigma\leq \frac{2}{3}$, and we set
$$b=\frac{2\sigma-1}{\sigma}(8-7\lambda)-4K(\lambda-1),$$
and
\[a_{1}=9(\lambda+1)^{2},\]
\[a_{2}=9(\lambda+1)^{2}-\frac{12(3-4\sigma)}{2\sigma-1}b,\]
\[a_{3}=\frac{350-444\sigma}{2\sigma-1}(\lambda-1)+\frac{8(3-4\sigma)}{2\sigma-1}(23\lambda-\frac{41}{2})b-\frac{8(3-4\sigma)(\sigma+1)}{(2\sigma-1)^{2}}b^{2}.\]
\[t_{0}=\frac{a_{2}+\sqrt{a_{2}^{2}+4a_{1}a_{3}}}{2a_{1}}, \quad \delta=\max\{t_{0}, \frac{13\lambda-10}{3(\lambda+2)}\}\]
If we assume that $|A|^{2}\leq \sigma |H|^{2}+bk_{1}$ and $\cos\alpha\geq \sqrt{\delta}$ holds on the initial surface, then it remains true along the symplectic mean curvature flow. Furthermore, the symplectic flow exists for long time and converges to a holomorphic curve at infinity.
\end{theorem}
\begin{remark}
If $\lambda=1$, then $\frac{13\lambda-10}{3(\lambda+2)}=\frac{1}{3}$ and
\[t_{0}=\frac{7\sigma-3+\sqrt{(7\sigma-3)^{2}+4(3-4\sigma)(3\sigma-2)}}{6\sigma}\leq \frac{7\sigma-3}{3\sigma}.\]
Hence even in the case of $\lambda=1$, the above theorem improving Han-Li-Yang's Theorem (Theorem 1.3).
\end{remark}
\begin{remark}
By Lemma \ref{BSC}, if $\lambda<\frac{3}{2}$, then the sectional curvature of $M$ is positive, implies that the bisectional curvature is positive. By Frankel conjecture, which was proved by Siu-Yau [SY] using harmonic maps and by Mori [M] via algebraic methods, then the K\"ahler surface is biholomorphic to $\mathbb{C}\mathbb{P}^{2}$.
\end{remark}
\section{Preliminaries}
\subsection{Evolution Equations}
Suppose that $\Sigma$ is a sub manifold in a Riemannian manifold $M$, we choose an orthonormal basis $\{e_{i}\}$ for $T\Sigma$ and $\{e_{\alpha}\}$ for $N\Sigma$.
Given an immersed $F_{0}:\Sigma\rightarrow M$, we consider a one parameter family of smooth maps $F_{t}=F(\cdot, t):\Sigma\rightarrow M$ with corresponding images $\Sigma_{t}=F_{t}(\Sigma)$ immersed in $M$ and $F$ satisfies the mean curvature flow equation:
\begin{equation}
\begin{cases}
\frac{\partial}{\partial t}F(x,t)=H(x,t)\\
F(x,0)=F_{0}(x)
\end{cases}
\end{equation}

Recall the evolution equation for the second fundamental form $h_{ij}^{\alpha}$ and $|A|^{2}$ along the mean curvature flow (see[CJ1], [HLY], [S1], [W]).
\begin{lemma}
\begin{eqnarray}
\begin{aligned}
\frac{\partial}{\partial t} h_{ij}^{\alpha}=&\Delta h_{ij}^{\alpha}+(\overline{\nabla}_{k}Rm)_{\alpha ijk}+(\overline{\nabla}_{j}Rm)_{\alpha kik}-2R_{lijk}h_{lk}^{\alpha}\\
&+2R_{\alpha\beta jk}h_{ik}^{\beta}+2R_{\alpha\beta ik}h_{jk}^{\beta}-R_{lkik}h_{lj}^{\alpha}-R_{lkjk}h_{il}^{\alpha}\\
&+R_{\alpha k\beta k}h_{ij}^{\beta}-H^{\beta}(h_{ik}^{\beta}h_{jk}^{\alpha}+h_{jk}^{\beta}h_{ik}^{\alpha})+h_{im}^{\alpha}h_{mk}^{\beta}h_{kj}^{\beta}\\
&-2h_{im}^{\beta}h_{mk}^{\alpha}h_{kj}^{\beta}+h_{ik}^{\beta}h_{km}^{\beta}h_{mj}^{\alpha}+h_{km}^{\alpha}h_{mk}^{\beta}h_{ij}^{\beta}+h_{ij}^{\beta}<e_{\beta}, \overline{\nabla}_{H}e_{\alpha}>,
\end{aligned}
\end{eqnarray}
where $R_{ABCD}$ is the curvature tensor of $M$ and $\overline{\nabla}$ is the covariant derivative of $M$. Therefore
\begin{eqnarray}\label{A}
\begin{aligned}
\frac{\partial}{\partial t}|A|^{2}=&\Delta A-2|\nabla A|^{2}+[(\overline{\nabla}_{k}Rm)_{\alpha ijk}+(\overline{\nabla}_{j}Rm)_{\alpha kik}]h_{ij}^{\alpha}\\
&-4R_{lijk}h_{lk}^{\alpha}h_{ij}^{\alpha}+8R_{\alpha\beta jk}h_{ik}^{\beta}h_{ij}^{\alpha}-4R_{lkik}h_{lj}^{\alpha}h_{ij}^{\alpha}+2R_{\alpha k\beta k}h_{ij}^{\beta}h_{ij}^{\alpha}+2P_{1}+2P_{2}
\end{aligned}
\end{eqnarray}
where
\begin{equation*}
P_{1}=\Sigma_{\alpha,\beta,i,j}(\Sigma_{k}(h_{ik}^{\alpha}h_{jk}^{\beta}-h_{jk}^{\alpha}h_{ik}^{\beta}))^{2},
\end{equation*}
\begin{equation*}
P_{2}=\Sigma_{\alpha,\beta}(\Sigma_{i,j}h_{ij}^{\alpha}h_{ij}^{\beta})^{2}.
\end{equation*}
\begin{equation}\label{H}
\frac{\partial}{\partial t}|H|^{2}=\Delta |H|^{2}-2|\nabla H|^{2}+2R_{\alpha k\beta k}H^{\alpha}H^{\beta}+2P_{3}
\end{equation}
where
\begin{equation*}
P_{3}=\Sigma_{i,j}(\Sigma_{\alpha}H^{\alpha}h_{ij}^{\alpha})^{2}.
\end{equation*}
\end{lemma}
Choose an orthonormal basis $\{e_{1},e_{2},e_{3},e_{4}\}$ on $(M,\overline{g})$ along $\Sigma_{t}$ such that
$\{e_{1},e_{2}\}$ is the frame of the tangent bundle $T \Sigma_{t}$ and $\{e_{3},e_{4}\}$ is the frame of the normal bundle
$N\Sigma_{t}$. Then along the surface $\Sigma_{t}$, we can takes the complex structure on $M$ as the form (cf. [HLY])
\begin{equation}\label{J0}
J=\left(
\begin{array}{cccc}
0 & \cos\alpha & y & z \\
-\cos\alpha & 0 & -z & y\\
-y & z & 0 & -\cos\alpha\\
-z & -y & \cos\alpha & 0
\end{array}
\right)
\end{equation}
or
\begin{equation}\label{J}
J=\left(
\begin{array}{cccc}
0 & \cos\alpha & y & z \\
-\cos\alpha & 0 & z & -y\\
-y & -z & 0 & \cos\alpha\\
-z & y & -\cos\alpha & 0
\end{array}
\right)
\end{equation}
Since K\"ahler form is self-dual, then $J$ must be the form (\ref{J}).

\begin{remark}
In fact, the above argument also shows that the K\"ahler form is self-dual. If $J$ is the form (\ref{J0}), then the K\"ahler form is anti-self-dual, i.e., $*\omega=-\omega$, it is impossible for K\"ahler form. Hence $J$ must be the form (\ref{J}), then the K\"ahler form $\omega$ must be self-dual.
\end{remark}

Recall the evolution equation of the K\"abler angle $\cos\alpha$ (cf. [CL1], [HL2]),
\begin{lemma}
The evolution equation for $\cos\alpha$ along $\Sigma_{t}$ is
\begin{equation}\label{EvolutionCos}
(\frac{\partial}{\partial t}-\Delta)\cos\alpha =|\overline{\nabla}J_{\Sigma_{t}}|^{2}\cos\alpha +\sin^{2}\alpha Ric(Je_1, e_2).
\end{equation}
\end{lemma}
Here
\begin{equation}
|\overline{\nabla}J_{\Sigma_{t}}|^{2}=|h^{4}_{1k}+h^{3}_{2k}|^{2}+|h^{4}_{2k}-h^{3}_{1k}|^{2}.
\end{equation}

$|\overline{\nabla}J_{\Sigma_{t}}|^{2}$ is independent of the choice of the frame and only depend on the orientation of the frame. It is proved in ([CL1], [HL1]) that
\begin{equation}
|\overline{\nabla}J_{\Sigma_{t}}|^{2}\geq \frac{1}{2}|H|^{2}
\end{equation}
and
\begin{equation}
|\nabla\cos\alpha|^{2}\leq \sin^{2}\alpha|\overline{\nabla}J_{\Sigma_{t}}|^{2}.
\end{equation}

\subsection{Curvatures}
In this subsection, we first recall the definitions of Riemannian curvature and the holomorphic sectional curvature; secondly, we give some estimates of Riemannian curvatures by holomorphic sectional curvature.

The Riemann curvature tensor $R$ of $(M,g)$ is defined by
\begin{equation*}
R(X, Y, Z, W)=-g(\nabla_{X}\nabla_{Y}Z-\nabla_{Y}\nabla_{X}Z-\nabla_{[X,Y]}Z, W)
\end{equation*}
for any vector fields $X, Y, Z,W$.

Set $R(X,Y)=R(X,Y,X,Y)$ and $R(X)=R(X,JX)$. Fix a point $p\in M$ and a two-dimensional plane $\Pi \subset T_{p}M$. The sectional curvature of $\Pi$ is defined by
\begin{equation*}
K(\Pi)=\frac{R(X,Y)}{g(X,X)g(Y,Y)-g(X,Y)^{2}}
\end{equation*}
where $\{X,Y\}$ is a basis of $\Pi$, we also denote it by $K(X,Y)$. For a K\"abler manifold $(M,g,J)$, if the two-dimensional plane $\Pi$ is spanned  by $\{X,JX\}$, i.e., $\Pi$ is a holomorphic plane, then the sectional curvature of $\Pi$ is called a holomorphic sectional curvature of $\Pi$, we denote it by $K(X)$, where $\{X,JX\}$ is a basis of $\Pi$. Then
\begin{equation*}
K(X)=\frac{R(X)}{g(X,X)^{2}}.
\end{equation*}

It is well known that we can express the sectional curvatures by holomorphic sectional curvatures.
\begin{theorem}
The sectional curvatures of $M$ can be determined by holomorphic sectional curvatures by
\begin{eqnarray}\label{2.1}
\begin{aligned}
R(X,Y)=&\frac{1}{32}[3R(X+JY)+3R(X-JY)-R(X+Y)-R(X-Y)\\
&-4R(X)-4R(Y)].
\end{aligned}
\end{eqnarray}
\end{theorem}
\begin{theorem}\label{2.2}
For any vector fields $X,Y$ and $Z$ on $M$,
\begin{equation}
R(X,Y,X,Z)=\frac{1}{4}(R(X,Y+Z)-R(X,Y-Z)).
\end{equation}
\end{theorem}
\begin{lemma}\label{BSC}
For any two orthogonal vectors $X$ and $Y$, set $|X|^{2}=a, |Y|^{2}=b, <JX,Y>=x$, then
\begin{equation}
R(X,Y)\leq \frac{1}{16}[(3(a+b)^{2}+12x^{2})k_{2}-(3a^{2}+3b^{2}+2ab)k_{1}]
\end{equation}
and
\begin{equation}
R(X,Y)\geq \frac{1}{16}[(3(a+b)^{2}+12x^{2})k_{1}-(3a^{2}+3b^{2}+2ab)k_{2}]
\end{equation}
\end{lemma}
\begin{proof}
Since
\begin{equation}
<X+JY, X+JY>=|X|^{2}-2<JX,Y>+|Y|^{2}=a+b-2x,
\end{equation}
and
\begin{equation}
<X-JY,X-JY>=a+b+2x, <X+Y,X+Y>=<X-Y,X-Y>=a+b,
\end{equation}
by (\ref{2.1}), we have
\begin{eqnarray}
\begin{aligned}
&\quad R(X,Y)\\
&\leq\frac{1}{32}[3(a+b-2x)^{2}k_{2}+3(a+b+2x)^{2}k_{2}-(a+b)^{2}k_{1}-(a+b)^{2}k_{1}-4a^{2}k_{1}-4b^{2}k_{1}]\\
&=\frac{1}{16}[(3(a+b)^{2}+12x^{2})k_{2}-(3a^{2}+3b^{2}+2ab)k_{1}]
\end{aligned}
\end{eqnarray}
and similarly
\begin{equation}
R(X,Y)\geq \frac{1}{16}[(3(a+b)^{2}+12x^{2})k_{1}-(3a^{2}+3b^{2}+2ab)k_{2}]
\end{equation}
\end{proof}
\begin{lemma}\label{ECO}
For the orthonormal basis $\{e_{1},e_{2},e_{3},e_{4}\}$ on $(M,g)$ along $\Sigma_{t}$ and takes the form $J$ as (\ref{J}). Then we have  the following estimates:
\begin{itemize}
\item[1)] $\frac{1}{4}[(3+3\cos^{2}\alpha)k_{1}-2k_{2}]\leq R_{1212}\leq \frac{1}{4}[(3+3\cos^{2}\alpha)k_{2}-2k_{1}]$;
\item[2)] $\frac{1}{4}[(3+3\cos^{2}\alpha)k_{1}-2k_{2}]\leq R_{3434}\leq \frac{1}{4}[(3+3\cos^{2}\alpha)k_{2}-2k_{1}]$;
\item[3)] $\frac{1}{4}[(3+3y^{2})k_{1}-2k_{2}]\leq R_{1313}\leq \frac{1}{4}[(3+3y^{2})k_{2}-2k_{1}]$;
\item[4)] $\frac{1}{4}[(3+3y^{2})k_{1}-2k_{2}]\leq R_{2424}\leq \frac{1}{4}[(3+3y^{2})k_{2}-2k_{1}]$;
\item[5)] $\frac{1}{4}[(3+3z^{2})k_{1}-2k_{2}]\leq R_{1414}\leq \frac{1}{4}[(3+3z^{2})k_{2}-2k_{1}]$;
\item[6)] $\frac{1}{4}[(3+3z^{2})k_{1}-2k_{2}]\leq R_{2323}\leq \frac{1}{4}[(3+3z^{2})k_{2}-2k_{1}]$;
\item[7)] $\frac{1}{32}[(23+6(\cos\alpha+y)^{2})k_{1}-(23+6(\cos\alpha-y)^{2})k_{2}]\leq R_{2131} \leq \frac{1}{32}[(23+6(\cos\alpha+y)^{2})k_{2}-(23+6(\cos\alpha-y)^{2})k_{1}]$;
\item[8)] $\frac{1}{32}[(23+6(\cos\alpha-y)^{2})k_{1}-(23+6(\cos\alpha+y)^{2})k_{2}]\leq R_{2434} \leq \frac{1}{32}[(23+6(\cos\alpha-y)^{2})k_{2}-(23+6(\cos\alpha+y)^{2})k_{1}]$;
\item[9)] $\frac{1}{32}[(23+6(\cos\alpha+y)^{2})k_{1}-(23+6(\cos\alpha-y)^{2})k_{2}]\leq R_{1242} \leq \frac{1}{32}[(23+6(\cos\alpha+y)^{2})k_{2}-(23+6(\cos\alpha-y)^{2})k_{1}]$;
\item[10)] $\frac{1}{32}[(23+6(\cos\alpha-z)^{2})k_{1}-(23+6(\cos\alpha+z)^{2})k_{2}]\leq R_{1232}\leq \frac{1}{32}[(23+6(\cos\alpha-z)^{2})k_{2}-(23+6(\cos\alpha+z)^{2})k_{1}]$;
\item[11)] $\frac{1}{32}[(23+6(\cos\alpha+z)^{2})k_{1}-(23+6(\cos\alpha-z)^{2})k_{2}]\leq R_{2141}\leq \frac{1}{32}[(23+6(\cos\alpha+z)^{2})k_{2}-(23+6(\cos\alpha-z)^{2})k_{1}]$;
\item[12)] $\frac{1}{32}[(23+6(y+z)^{2})k_{1}-(23+6(y-z)^{2})k_{2}]\leq R_{3141}\leq \frac{1}{32}[(23+6(y+z)^{2})k_{2}-(23+6(y-z)^{2})k_{1}]$;
\item[13)] $\frac{1}{32}[(23+6(y-z)^{2})k_{1}-(23+6(y+z)^{2})k_{2}]\leq R_{3242}\leq \frac{1}{32}[(23+6(y-z)^{2})k_{2}-(23+6(y+z)^{2})k_{1}]$;
\item[14)] $\frac{1}{12}[(10+6\cos^{2}\alpha)k_{1}-(10+3\sin^{2}\alpha)k_{2}]\leq R_{1234}\leq \frac{1}{12}[(10+6\cos^{2}\alpha)k_{2}-(10+3\sin^{2}\alpha)k_{1}].$
\item[15)] $\frac{1}{2}(6k_{1}-3k_{2})\leq R_{ii}\leq \frac{1}{2}(6k_{2}-3k_{1}) (1\leq i\leq 4).$
\item[16)] $|R_{34}|\leq \frac{29-6\cos^{2}\alpha}{16}(k_{2}-k_{1}).$
\end{itemize}
\end{lemma}
\begin{proof}
By (\ref{J}), we have
\begin{itemize}
\item $Je_{1}=\cos\alpha e_{2}+y e_{3}+z e_{4},$
\item $Je_{2}=-\cos\alpha e_{1}+z e_{3}-y e_{4},$
\item $Je_{3}=-y e_{1}-z e_{2}+\cos\alpha e_{4},$
\item $Je_{4}=-ze_{1}+y e_{2}-\cos\alpha e_{3}.$
\end{itemize}
Hence $<Je_{1},e_{2}>=<Je_{3},e_{4}>=\cos\alpha$, $<Je_{1}, e_{3}>=<Je_{4},e_{2}>=y$, $<Je_{1},e_{4}>=<Je_{2},e_{3}>=z$, then by Lemma \ref{BSC}, we get 1)-6).

By Theorem \ref{2.2},
\begin{equation}
R_{1213}=\frac{1}{4}(R(e_{1}, e_{2}+e_{3})-R(e_{1}, e_{2}-e_{3})).
\end{equation}
Since $Je_{1}=\cos\alpha e_{2}+y e_{3}+z e_{4}$, $<Je_{1}, e_{2}+e_{3}>=\cos\alpha+y$ and $<Je_{1}, e_{2}-e_{3}>=\cos\alpha-y$. Then by Lemma \ref{BSC},
\begin{equation}
\frac{1}{16}[(27+12(\cos\alpha+y)^{2})k_{1}-19k_{2}]\leq R(e_{1}, e_{2}+e_{3})\leq \frac{1}{16}[(27+12(\cos\alpha+y)^{2})k_{2}-19k_{1}],
\end{equation}
and
\begin{equation}
\frac{1}{16}[(27+12(\cos\alpha-y)^{2})k_{1}-19k_{2}]\leq R(e_{1}, e_{2}-e_{3})\leq \frac{1}{16}[(27+12(\cos\alpha-y)^{2})k_{2}-19k_{1}],
\end{equation}
Hence
\begin{eqnarray}
\begin{aligned}
R_{1213}&\leq\frac{1}{64}[(46+12(\cos\alpha+y)^{2})k_{2}-(46+12(\cos\alpha-y)^{2})k_{1}]\\
&=\frac{1}{32}[(23+6(\cos\alpha+y)^{2})k_{2}-(23+6(\cos\alpha-y)^{2})k_{1}]
\end{aligned}
\end{eqnarray}
and
\begin{equation}
R_{1213}\geq\frac{1}{32}[(23+6(\cos\alpha+y)^{2})k_{1}-(23+6(\cos\alpha-y)^{2})k_{2}]
\end{equation}
Hence we obtain 7).

Using Theorem \ref{2.2}, Lemma \ref{BSC} and the same argument as in the proof of 7) , we can obtain 8)-13).

It is easy to check the following identity
\begin{eqnarray}
\begin{aligned}
&24R_{1234}\\
=&R(e_{1}+e_{3},e_{2}+e_{4})-R(e_{1}+e_{3},e_{2}-e_{4})-R(e_{1}-e_{3},e_{2}+e_{4})+R(e_{1}-e_{3},e_{2}-e_{4})\\
&-R(e_{1}+e_{4},e_{2}+e_{3})+R(e_{1}+e_{4},e_{2}-e_{3})+R(e_{1}-e_{4},e_{2}+e_{3})-R(e_{1}-e_{4},e_{2}-e_{3})
\end{aligned}
\end{eqnarray}
Then by Lemma \ref{BSC}, we get
\begin{equation}
2[(10+6\cos^{2}\alpha)k_{1}-(10+3\sin^{2}\alpha)k_{2}]\leq 24R_{1234} \leq 2[(10+6\cos^{2}\alpha)k_{2}-(10+3\sin^{2}\alpha)k_{1}].
\end{equation}
Hence
\begin{equation}
\frac{1}{12}[(10+6\cos^{2}\alpha)k_{1}-(10+3\sin^{2}\alpha)k_{2}]\leq R_{1234}\leq \frac{1}{12}[(10+6\cos^{2}\alpha)k_{2}-(10+3\sin^{2}\alpha)k_{1}]
\end{equation}
Therefore we get 14).

By 1)-14), it is easy to get 15) and 16).
\end{proof}

\section{Lower bound along the mean curvature flow}
In this section, we following the argument in [LY] to prove the first main theorem of this paper, which improving the main theorem in [LY].
\begin{theorem}
Suppose $M$ is a K\"ahler surface with positive holomorphic sectional curvatures. If $1\leq \lambda <2$ and $\cos\alpha(\cdot,0)\geq \delta >\frac{29(\lambda-1)}{\sqrt{(48-24\lambda)^{2}+(29\lambda-29)^{2}}}$, then
along the flow
\begin{equation}
(\frac{\partial}{\partial t}-\Delta)\cos\alpha \geq |\overline{\nabla}J_{\Sigma_{t}}|^{2}\cos\alpha +C\sin^{2}\alpha,
\end{equation}
where $C$ is a positive constant depending only on $k_1,k_2$ and $\delta$. As a corollary, $\min_{\Sigma_{t}}\cos\alpha$ is increasing with respect to $t$. In particular, at each time $t$, $\Sigma_{t}$ is symplectic.
\end{theorem}

\begin{proof}
For simplicity, we can take $y=\sin\alpha, z=0$ in the form of $J$. In order to prove this theorem, we need to estimate ${\rm Ric}(Je_{1}, e_{2})$. then
\begin{eqnarray}
\begin{aligned}
{\rm Ric}(Je_{1}, e_{2})&=\sum_{i=1}^{4}R(Je_{1}, e_{i}, e_{2}, e_{i})\\
&=\sum_{i=1}^{4}R(\cos\alpha e_{2}+\sin\alpha e_{3}, e_{i}, e_{2}, e_{i})\\
&=\cos\alpha R_{22}+\sin\alpha(R_{1213}+R_{4243})\\
&=\cos\alpha R_{22}+\sin\alpha R_{23}.
\end{aligned}
\end{eqnarray}
By Lemma \ref{ECO}, we have
\begin{equation}
R_{22}=R_{1212}+R_{3232}+R_{4242} \geq 3k_{1}-\frac{3}{2}k_{2},
\end{equation}
and
\begin{equation}
|R_{23}|=|R_{1213}+R_{4243}|\leq\frac{29}{16}(k_{2}-k_{1}).
\end{equation}
Hence we have
\begin{eqnarray}
\begin{aligned}
{\rm Ric}(Je_{1}, e_{2})&\geq \cos\alpha (3k_{1}-\frac{3}{2}k_{2})-\sqrt{1-\cos^{2}\alpha}\frac{29}{16}(k_{2}-k_{1})\\
&=(3\cos\alpha+\frac{29}{16}\sqrt{1-\cos^{2}\alpha})k_{1}-(\frac{3}{2}\cos\alpha+\frac{29}{16}\sqrt{1-\cos^{2}\alpha})k_{2}.
\end{aligned}
\end{eqnarray}
It follows that if $1\leq \lambda<2$ and $\cos\alpha >\frac{29(\lambda-1)}{\sqrt{(48-24\lambda)^{2}+(29\lambda-29)^{2}}}$, then the right hand side of (28) is positive. Hence we obtain the Theorem 4.1.
\end{proof}
\begin{remark}
Comparing to the proof in [LY], we do more carefully to estimate the term $R_{1213}+R_{4243}$ by Lemma \ref{ECO}.
\end{remark}

As same as Corollary 1.2 and Theorem 1.3 in [LY], we also have the following Corollary and Theorem.

Arguing as in [CW] by strong maximum principle, we have
\begin{cor}
Suppose $M$ is a K\"ahler surface with positive holomorphic sectional curvatures and $1\leq \lambda<2$, then every symplectic minimal surface satisfying
\begin{equation*}
\cos\alpha>\frac{29(\lambda-1)}{\sqrt{(48-24\lambda)^{2}+(29\lambda-29)^{2}}}
\end{equation*}
in $M$ is a holomorphic curve.
\end{cor}

Arguing exactly in the same way as in [CL1] or [W], we have
\begin{theorem}
Under the same condition of Theorem 3.1, then the symplectic mean curvature flow has no type I singularity at any $T>0$.
\end{theorem}

\section{When $\cos\alpha$ is close to $1$}
In this section, we will use the same argument of Han and Li [HL0], to prove that  K\"ahler manifold $M$ with positive holomorphic sectional curvature and $1\leq \lambda<2$, when $\cos\alpha$ is close enough to 1, then the mean curvature flow exists globally and converges to a holomorphic curve.

\begin{prop}
Suppose that $M$ is a K\"ahler surface with positive holomorphic sectional curvature and $1\leq \lambda<2$. Let $\alpha$ be the K\"ahler angle of the surface $\Sigma_{t}$ which evolves by the mean curvature flow. Suppose that $\cos\alpha(\cdot,0)>\frac{58(\lambda-1)}{\sqrt{(48-24)^{2}+(58\lambda-58)^{2}}}$. Then
\begin{eqnarray}
\begin{aligned}
&\int_{\Sigma_{t}}\frac{\sin^{2}\alpha}{\cos\alpha}d\mu_{t}\leq C_{0}e^{-\frac{3}{4}(2-\lambda)k_{1}t},\\
&\int_{t}^{t+1}\int_{\Sigma_{t}}|H|^{2}d\mu_{t}dt\leq C_{0}e^{-\frac{3}{4}(2-\lambda)k_{1}t},
\end{aligned}
\end{eqnarray}
where $C_{0}$ is a constant which depends only on the initial surface, $C_{0}=\int_{\Sigma_{0}}\frac{\sin^{2}\alpha(x,0)}{\cos\alpha(x,0)}d\mu_{0}$.
\end{prop}
\begin{proof}
By Theorem 3.1, we know $\cos\alpha(\cdot,t)>\frac{58(\lambda-1)}{\sqrt{(48-24)^{2}+(58\lambda-58)^{2}}}$ is preserved along the mean curvature flow. Since $\cos\alpha>\frac{58(\lambda-1)}{\sqrt{(48-24)^{2}+(58\lambda-58)^{2}}}$, then by (32), we have
\[Ric(Je_{1},e_{2})>\frac{3}{4}(2-\lambda)k_{1}\cos\alpha.\]
Hence
\[(\frac{\partial}{\partial t}-\Delta)\cos\alpha>|\overline{\nabla}J_{\Sigma_{t}}|^{2}\cos\alpha+\frac{3}{4}(2-\lambda)k_{1}\cos\alpha \sin^{2}\alpha\geq \frac{3}{4}(2-\lambda)k_{1}\cos\alpha \sin^{2}\alpha.\]

From the evolution equation of $\cos\alpha$, we have
\begin{eqnarray}
\begin{aligned}
&(\frac{\partial}{\partial t}-\Delta)\frac{1}{\cos\alpha}\\
=&-\frac{(\frac{\partial}{\partial t}-\Delta)\cos\alpha}{\cos^{2}\alpha}-\frac{2|\nabla\cos\alpha|^{2}}{\cos^{3}\alpha}\\
\leq &-\frac{3}{4}\frac{(2-\lambda)k_{1}\sin^{2}\alpha}{\cos\alpha}-\frac{2|\nabla\cos\alpha|^{2}}{\cos^{3}\alpha}
\end{aligned}
\end{eqnarray}
From the proof of Proposition 2.1 in [HL0], we have $\int_{\Sigma_{t}}\cos\alpha d\mu_{t}=\int_{\Sigma_{t}}\omega$ is constant under the continuous deformation in $t$.

We therefore have
\begin{eqnarray}
\begin{aligned}
&\frac{\partial}{\partial t}\int_{\Sigma_{t}}\frac{\sin^{2}\alpha}{\cos\alpha}d\mu_{t}\\
=&\frac{\partial}{\partial t}\int_{\Sigma_{t}}(\frac{1}{\cos\alpha}-\cos\alpha)d\mu_{t}=\frac{\partial}{\partial t}\int_{\Sigma_{t}}\frac{1}{\cos\alpha}d\mu_{t}\\
=&\int_{\Sigma_{t}}(\frac{\partial}{\partial t}-\Delta)\frac{1}{\cos\alpha}d\mu_{t}-\int_{\Sigma_{t}}\frac{|H|^{2}}{\cos\alpha}d\mu_{t}\\
\leq &-\frac{3}{4}(2-\lambda)k_{1}\int_{\Sigma_{t}}\frac{\sin^{2}\alpha}{\cos\alpha}d\mu_{t}-\int_{\Sigma_{t}}\frac{|H|^{2}}{\cos\alpha}d\mu_{t}.
\end{aligned}
\end{eqnarray}
So
\[\frac{\partial}{\partial t}\int_{\Sigma_{t}}\frac{\sin^{2}\alpha}{\cos\alpha}d\mu_{t}\leq -\frac{3}{4}(2-\lambda)k_{1}\int_{\Sigma_{t}}\frac{\sin^{2}\alpha}{\cos\alpha}d\mu_{t}.\]
Then it is easy to get that
\[\int_{\Sigma_{t}}\frac{\sin^{2}\alpha}{\cos\alpha}d\mu_{t}\leq e^{-\frac{3}{4}(2-\lambda)k_{1}t}\int_{\Sigma_{0}}\frac{\sin^{2}\alpha}{\cos\alpha}d\mu_{0},\]
that is,
\[\int_{\Sigma_{t}}\frac{\sin^{2}\alpha}{\cos\alpha}d\mu_{t}\leq C_{0}e^{-\frac{3}{4}(2-\lambda)k_{1}t}.\]

It follows from (35) that
\[\int_{\Sigma_{t}}|H|^{2}d\mu_{t}\leq-\frac{\partial}{\partial t}\int_{\Sigma_{t}}\frac{\sin^{2}\alpha}{\cos\alpha}d\mu_{t}.\]
Integrating the above inequality from $t$ to $t+1$, we obtain that
\[\int_{t}^{t+1}\int_{\Sigma_{t}}|H|^{2}d\mu_{t}dt\leq \int_{\Sigma_{t}}\frac{\sin^{2}\alpha}{\cos\alpha}d\mu_{t}\leq C_{0}e^{-\frac{3}{4}(2-\lambda)k_{1}t}.\]
This proves the proposition.
\end{proof}

We derive an $L^{1}$-estimation of the mean curvature vector on the time space.
\begin{prop}
Suppose that $M$ is a K\"ahler surface with positive holomorphic sectional curvature and $1\leq \lambda<2$. Let $\alpha$ be the K\"ahler angle of the surface $\Sigma_{t}$ which evolves by the mean curvature flow. Suppose that $\cos\alpha(\cdot,0)>\frac{58(\lambda-1)}{\sqrt{(48-24)^{2}+(58\lambda-58)^{2}}}$.Then
\begin{equation}
\int_{0}^{T}\int_{\Sigma_{t}}|H|d\mu_{t}dt\leq (C_{0})^{1/2}\frac{Area(\Sigma_{0})^{1/2}}{1-e^{-\frac{3}{8}(2-\lambda)k_{1}}},
\end{equation}
where the constant $C_{0}$ is defined in Proposition 4.1.
\end{prop}
\begin{proof}
We have
\begin{eqnarray*}
\begin{aligned}
\int_{0}^{T}\int_{\Sigma_{t}}|H|d\mu_{t}dt&\leq\sum_{k=0}^{[T]}\int_{k}^{k+1}\int_{\Sigma_{t}}|H|d\mu_{t}dt\\
&\leq\sum_{k=0}^{[T]}(\int_{k}^{k+1}\int_{\Sigma_{t}}|H|^{2}d\mu_{t}dt)^{1/2}(\int_{k}^{k+1}Area(\Sigma_{t}))^{1/2}\\
&\leq Area(\Sigma_{0})^{1/2}\sum_{k=0}^{[T]}(\int_{k}^{k+1}\int_{\Sigma_{t}}|H|^{2}d\mu_{t}dt)^{1/2}\\
&\leq C_{0}^{1/2}Area(\Sigma_{0})^{1/2}\sum_{k=0}^{[T]}e^{-\frac{3}{8}(2-\lambda)k_{1}k}\\
&\leq (C_{0})^{1/2}\frac{Area(\Sigma_{0})^{1/2}}{1-e^{-\frac{3}{8}(2-\lambda)k_{1}}}.
\end{aligned}
\end{eqnarray*}
This proves the proposition.
\end{proof}
\begin{remark}
Han and Li [HL0] proved the above propositions in the case of K\"ahler-Einstein manifold $M$ with positive scalar curvature $R$.
\begin{eqnarray*}
\begin{aligned}
&\int_{\Sigma_{t}}\frac{\sin^{2}\alpha}{\cos\alpha}d\mu_{t}\leq C_{0}e^{-Rt},\\
&\int_{t}^{t+1}\int_{\Sigma_{t}}|H|^{2}d\mu_{t}dt\leq C_{0}e^{-Rt},\\
&\int_{0}^{T}\int_{\Sigma_{t}}|H|d\mu_{t}dt\leq (C_{0})^{1/2}\frac{Area(\Sigma_{0})^{1/2}}{1-e^{-R/2}}
\end{aligned}
\end{eqnarray*}
\end{remark}
We recall White's local regularity theorem.

Let $H(X,X_{0},t)$ be the backward heat kernel on $\mathbb{R}^{4}$. Define
\begin{equation}
\rho(X,t)=4\pi(t_{0}-t)H(X,X_{0},t)=\frac{1}{4\pi(t_{0}-t)} \exp(-\frac{|X-X_{0}|^{2}}{4(t_{0}-t)})
\end{equation}
for $t<t_{0}$. Let $i_{M}$ be the injective radius of $M^{4}$. We choose a cutoff function $\phi\in C_{0}^{\infty}(B_{2r}(X_{0}))$ with $\phi\equiv 1$ in $B_{r}(X_{0})$, where $X_{0}\in M, 0<2r<i_{M}$. Choose normal coordinates in $B_{2r}(X_{0})$ and express $F$ using the coordinates $(F^{1}, F^{2}, F^{3}, F^{4})$ as a surface in $\mathbb{R}^{4}$. The parabolic density of the mean curvature flow is defined by
\begin{equation}
\Phi(X_{0}, t_{0}, t)=\int_{\Sigma_{t}}\phi(F)\rho(F,t)d\mu_{t}.
\end{equation}
The following local regularity theorem was proved by White (see [Wh, Theorems 3.1 and 4.1]).
\begin{theorem}
There is a positive constant $\epsilon_{0}>0$ such that if
\begin{equation}
\Phi(X_{0},t_{0}, t_{0}-r^{2})\leq 1+\epsilon_{0},
\end{equation}
then the second fundamental form $A(t)$ of $\Sigma_{t}$ in $M$ is bounded in $B_{r/2}(X_{0})$, that is,
\begin{equation}
\sup_{B_{r/2}\times (t_{0}-r^{2}/4, t_{0}]}|A|\leq C,
\end{equation}
 where $C$ is a positive constant depending only on $M$.
 \end{theorem}
 \begin{remark}
 Since $\Sigma_{0}$ is smooth, it is well known that
 \begin{equation*}
 \lim_{r\rightarrow 0}\int_{\Sigma_{0}}\phi(F)\frac{e^{-(|F-X_{0}|^{2}/4r^{2})}}{4\pi r^{2}}d\mu_{0}=1
 \end{equation*}
 for any $X_{0}\in \Sigma_{0}.$ So we can find a sufficiently small $r_{0}$ such that
 \begin{equation*}
 \int_{\Sigma_{0}}\phi(F)\frac{e^{-(|F-X_{0}|^{2}/4r_{0}^{2})}}{4\pi r_{0}^{2}}d\mu_{0}\leq 1+\frac{\epsilon_{0}}{2}
 \end{equation*}
 i.e.,
 \[\Phi(X_{0}, r_{0}^{2}, 0)\leq 1+\frac{\epsilon_{0}}{2}\]
 for all $X_{0}\in M$, where $\epsilon_{0}$ is the constant in White's theorem.
 \end{remark}

 We state the main theorem in this section as following.
 \begin{theorem}
Suppose that $M$ is a K\"ahler surface with positive holomorphic sectional curvature and $1\leq \lambda <2$. Let $\alpha$ be the K\"ahler angle of the surface $\Sigma_{t}$ which evolves by the mean curvature flow. Suppose that $\cos\alpha(\cdot,0)>\frac{58(\lambda-1)}{\sqrt{(48-24)^{2}+(58\lambda-58)^{2}}}$.Then there exists a sufficiently small constant $\epsilon_{1}$ such that if $C_{0}\leq \epsilon_{1}$ and $\epsilon_{1}$ satisfying
\[\epsilon_{1}\leq \frac{\pi^{2}\epsilon_{0}^{2}r_{0}^{6}(1-e^{-\frac{3}{8}(2-\lambda)k_{1}})^{2}}{4 Area(\Sigma_{0})},\]
there $C_{0}$ is defined in Proposition 4.1, $r_{0}$ is defined in Remark 4.1.1, and $\epsilon_{0}$ is the constant in White's theorem, the mean curvature flow with initial surface $\Sigma_{0}$ exists globally and it converges to a holomorphic curve.
\end{theorem}
 \begin{proof}
The argument is same as in Han and Li [HL0]. For convenience of readers, we provide the argument here.

Fix any positive number $T$. By the definition of $\Phi$, we have
\begin{equation}
\Phi(X_{0},t,t-r^{2})=\int_{\Sigma_{t-r^{2}}}\Phi(F)\frac{e^{-|F-X_{0}|^{2}/4r^{2}}}{4\pi r^{2}}d\mu_{t-r^{2}}.
\end{equation}
Differentiating the above equation with respect to $t$, we get that
\begin{eqnarray*}
\begin{aligned}
\frac{\partial}{\partial t}\Phi(X_{0},t,t-r^{2})&=\int_{\Sigma_{t-r^{2}}}<\nabla\phi, H>\frac{e^{-|F-X_{0}|^{2}/4r^{2}}}{4\pi r^{2}}d\mu_{t-r^{2}}\\
&-\int_{\Sigma_{t-r^{2}}}\frac{\Phi<F-X_{0}, H>}{8\pi r^{4}}e^{-|F-X_{0}|^{2}/4r^{2}}d\mu_{t-r^{2}}\\
&-\int_{\Sigma_{t-r^{2}}}\Phi(F)|H|^{2}\frac{e^{-|F-X_{0}|^{2}/4r^{2}}}{4\pi r^{2}}d\mu_{t-r^{2}}.
\end{aligned}
\end{eqnarray*}
Integrating the above equation from $r^{2}$ to $T$, and set $r=r_{0}$, then we get
\begin{eqnarray*}
\begin{aligned}
\Phi(X_{0}, T, T-r_{0}^{2})\leq &\Phi(X_{0}, r_{0}^{2}, 0)+\int_{r_{0}^{2}}^{T}\int_{\Sigma_{t-r_{0}^{2}}}|\nabla\phi| |H|\frac{e^{-|F-X_{0}|^{2}/4r_{0}^{2}}}{4\pi r_{0}^{2}}d\mu_{t-r_{0}^{2}}\\
&+\int_{r_{0}^{2}}^{T}\int_{\Sigma_{t-r_{0}^{2}}}\frac{\Phi|F-X_{0}||H|}{8\pi r_{0}^{4}}e^{-|F-X_{0}|^{2}/4r_{0}^{2}}d\mu_{t-r_{0}^{2}}.
\end{aligned}
\end{eqnarray*}
Using Remark 4.1.1, Proposition 4.1 and 4.2 and noting that we can choose $\phi$ such that $|\nabla\phi|\leq \frac{2}{r_{0}}$, we obtain that
\begin{eqnarray*}
\begin{aligned}
&\Phi(X_{0}, T, T-r_{0}^{2})\\
\leq&1+\frac{\epsilon_{0}}{2}+\frac{1}{2\pi r_{0}^{3}}(C_{0})^{1/2}\frac{Area(\Sigma_{0})^{1/2}}{1-e^{-\frac{3}{8}(2-\lambda)k_{1}}}+\frac{1}{2\pi r_{0}^{3}}(C_{0})^{1/2}\frac{Area(\Sigma_{0})^{1/2}}{1-e^{-\frac{3}{8}(2-\lambda)k_{1}}}\\
\leq &1+\epsilon_{0}.
\end{aligned}
\end{eqnarray*}
we have use that $e^{-x^{2}}x\leq e^{-1/2}\frac{\sqrt{2}}{2}<\frac{1}{2}$ in the last inequality.

Applying White's theorem we obtain a uniform estimate of the second fundamental form which implies the global existence and the convergence of the mean curvature flow. By Proposition 4.2, we have that
\begin{equation*}
\int_{\Sigma_{t}}\frac{\sin^{2}\alpha}{\cos\alpha}d\mu_{t}\leq C_{0}e^{-\frac{3}{4}(2-\lambda)k_{1}t}.
\end{equation*}
Let $t\rightarrow \infty$ and $(2-\lambda)k_{1}>0$, we get
\begin{equation*}
\int_{\Sigma_{\infty}}\frac{\sin^{2}\alpha}{\cos\alpha}d\mu_{\infty}=0.
\end{equation*}
It follows that $\cos\alpha_{\infty}=1$, that is, $\Sigma_{\infty}$ is a holomorphic curves. This proves the theorem.
\end{proof}
\begin{cor}
Suppose that $M$ is a K\"ahler surface with positive holomorphic sectional curvature and $1\leq \lambda <2$. Let $\alpha$ be the K\"ahler angle of the surface $\Sigma_{t}$ which evolves by the mean curvature flow. If there exists a positive constant $\delta$ such that
\[\cos\alpha(\cdot,0)\geq \delta > \frac{29(\lambda-1)}{\sqrt{(48-24\lambda)^{2}+(29\lambda-29)^{2}}}\]
and has the longtime existence. Then the mean curvature flow converges to a holomorphic curve at infinity.
\end{cor}
\begin{proof}
We just consider the case of $\lambda>1$.

Since $\delta > \frac{29(\lambda-1)}{\sqrt{(48-24\lambda)^{2}+(29\lambda-29)^{2}}}$,  we can choose $K=\frac{24(2-\lambda)}{\sqrt{\frac{1}{\delta^{2}}-1}(\lambda-1)}>29$, then
\[
\cos\alpha(\cdot,0)\geq  \frac{K(\lambda-1)}{\sqrt{(48-24\lambda)^{2}+(K\lambda-K)^{2}}}>\frac{29(\lambda-1)}{\sqrt{(48-24\lambda)^{2}+(29\lambda-29)^{2}}}.\]
Then we can obtain the result from the proof of Theorem 4.2.
\end{proof}
\section{Pinching estimate}
This section, we prove a theorem generalizing Theorem 3.2 in [HLY]. First we recall a lemma in [HLY]
\begin{lemma}\label{L3.1}
For any $\eta>0$ we have the inequality
\begin{equation}
|\nabla A|^{2}\geq (\frac{3}{4}-\eta)|\nabla H|^{2}-(\frac{1}{4}\eta^{-1}-1)|w|^{2},
\end{equation}
where $w_{i}^{\alpha}=\Sigma_{l}R_{\alpha lil}, |w^{\alpha}|^{2}=\Sigma_{i}|w_{i}^{\alpha}|^{2}$ and $|w|^{2}=\Sigma_{\alpha}|w^{\alpha}|^{2}.$
\end{lemma}
Now we can prove the following theorem.
\begin{theorem}
Suppose $\Sigma$ is a symplectic surface in the K\"ahler surface $(M,J,\omega,\overline{g})$ with positive holomorphic sectional curvature and $1\leq \lambda <1+\frac{1}{200}$, and $|\overline{\nabla}Rm|\leq Kk_{1}(\lambda-1)$ for a positive constant $K\leq\min\{2, 2k_{1}\}$. For any $\sigma$ satisfying $\frac{1}{2}+\frac{24(\lambda-1)}{1-34(\lambda-1)}<\sigma\leq \frac{2}{3}$, and we set
$$b=\frac{2\sigma-1}{\sigma}(8-7\lambda)-4K(\lambda-1),$$
and
\[a_{1}=9(\lambda+1)^{2},\]
\[a_{2}=9(\lambda+1)^{2}-\frac{12(3-4\sigma)}{2\sigma-1}b,\]
\[a_{3}=\frac{350-444\sigma}{2\sigma-1}(\lambda-1)+\frac{8(3-4\sigma)}{2\sigma-1}(23\lambda-\frac{41}{2})b-\frac{8(3-4\sigma)(\sigma+1)}{(2\sigma-1)^{2}}b^{2}.\]
\[t_{0}=\frac{a_{2}+\sqrt{a_{2}^{2}+4a_{1}a_{3}}}{2a_{1}}, \quad \delta=\max\{t_{0}, \frac{13\lambda-10}{3(\lambda+2)}\}\]
If we assume that $|A|^{2}\leq \sigma |H|^{2}+bk_{1}$ and $\cos\alpha\geq \sqrt{\delta}$ holds on the initial surface, then it remains true along the symplectic mean curvature flow.
\end{theorem}
\begin{proof}
Since we assume $\lambda\leq 1+\frac{1}{200}$, we have
\begin{equation*}
\frac{29(\lambda-1)}{\sqrt{(48-24\lambda)^{2}+(29\lambda-29)^{2}}}
<\frac{\sqrt{3}}{3}\leq \sqrt{\frac{13\lambda-10}{3(\lambda+2)}}\leq \sqrt{\delta}.
\end{equation*}
Then by Theorem 4.1 we know $\cos\alpha\geq \sqrt{\delta}$ is preserved by the mean curvature flow.

Since $|\overline{\nabla}Rm|\leq Kk_{1}(\lambda-1)$, Then
\begin{eqnarray*}
\begin{aligned}
\frac{\partial}{\partial t}|A|^{2}=&\Delta |A|^{2}+16Kk_{1}(\lambda-1)|A|-2|\nabla A|^{2}-4R_{lijk}h_{lk}^{\alpha}h_{ij}^{\alpha}+8R_{\alpha\beta jk}h_{ik}^{\beta}h_{ij}^{\alpha}\\
&-4R_{lkik}h_{lj}^{\alpha}h_{ij}^{\alpha}+2R_{\alpha k\beta k}h_{ij}^{\beta}h_{ij}^{\alpha}+2P_{1}+2P_{2}\\
\leq &\Delta |A|^{2}-2|\nabla A|^{2}-4R_{lijk}h_{lk}^{\alpha}h_{ij}^{\alpha}+8R_{\alpha\beta jk}h_{ik}^{\beta}h_{ij}^{\alpha}\\
&-4R_{lkik}h_{lj}^{\alpha}h_{ij}^{\alpha}+2R_{\alpha k\beta k}h_{ij}^{\beta}h_{ij}^{\alpha}+2P_{1}+2P_{2}+8Kk_{1}(\lambda-1)+8Kk_{1}(\lambda-1)|A|^{2}
\end{aligned}
\end{eqnarray*}
From the calculation in the proof of Theorem 3.2 in [HLY], we have
\begin{itemize}
\item $-4R_{lijk}h_{lk}^{\alpha}h_{ij}^{\alpha}=-4R_{1212}(|A|^{2}-|H|^{2}),$
\item $8R_{\alpha\beta jk}h_{ik}^{\beta}h_{ij}^{\alpha}=8R_{1234}(|A|^{2}-|\overline{\nabla}J_{\Sigma_{t}}|^{2}),$
\item $-4R_{lkik}h_{lj}^{\alpha}h_{ij}^{\alpha}=-R_{1212}|A|^{2},$
\item $2R_{\alpha k\beta k}h_{ij}^{\beta}h_{ij}^{\alpha}=2R_{33}(h_{ij}^{3})^{2}+2R_{44}(h_{ij}^{4})^{2}+4R_{34}h_{ij}^{3}h_{ij}^{4}-2R_{3434}|A|^{2}.$
\end{itemize}
Hence we have
\begin{eqnarray*}
\begin{aligned}
(\frac{\partial}{\partial t}-\Delta)|A|^{2}\leq&-2|\nabla A|^{2}+8(R_{1234}-R_{1212})|A|^{2}-2R_{3434}|A|^{2}+4R_{1212}|H|^{2}\\
&-8R_{1234}|\overline{\nabla}J_{\Sigma_{t}}|^{2}+2R_{33}(h_{ij}^{3})^{2}+2R_{44}(h_{ij}^{4})^{2}+4R_{34}h_{ij}^{3}h_{ij}^{4}\\
&+2P_{1}+2P_{2}+8Kk_{1}(\lambda-1)+8Kk_{1}(\lambda-1)|A|^{2}
\end{aligned}
\end{eqnarray*}

Since
\begin{equation*}
2R_{\alpha k\beta k}H^{\alpha}H^{\beta}=2R_{33}(H^{3})^{2}+2R_{44}(H^{4})^{2}+4R_{34}H^{3}H^{4}-2R_{3434}|H|^{2},
\end{equation*}
by (\ref{H}), we have
\begin{equation}
(\frac{\partial}{\partial t}-\Delta)|H|^{2}
=-2|\nabla H|^{2}+2R_{33}(H^{3})^{2}+2R_{44}(H^{4})^{2}+4R_{34}H^{3}H^{4}-2R_{3434}|H|^{2}+2P_{3}
\end{equation}

Hence
\begin{eqnarray}
\begin{aligned}
&(\frac{\partial}{\partial t}-\Delta)(|A|^{2}-\sigma |H|^{2})\\
\leq&-2(|\nabla A|^{2}-\sigma |\nabla H|^{2})+8(R_{1234}-R_{1212})(|A|^{2}-\sigma |H|^{2})+[8\sigma R_{1234}+(4-8\sigma)R_{1212}]|H|^{2}\\
&-8R_{1234}|\overline{\nabla}J_{\Sigma_{t}}|^{2}+2R_{33}(|A|^{2}-\sigma |H|^{2})+2(R_{44}-R_{33})(|h_{4}|^{2}-\sigma (H^{4})^{2})\\
&+4R_{34}(h_{ij}^{3}h_{ij}^{4}-\sigma H^{3}H^{4})-2R_{3434}(|A|^{2}-\sigma |H|^{2})\\
&+2P_{1}+2P_{2}-2\sigma P_{3}+8Kk_{1}(\lambda-1)+8Kk_{1}(\lambda-1)|A|^{2},
\end{aligned}
\end{eqnarray}
where $|h_{\alpha}|^{2}$ and $H^{\alpha}$ are defined by
\[|h_{\alpha}|^{2}=h^{\alpha}_{ij}h^{\alpha}_{ij}, \quad H^{\alpha}=\sum_{i}h^{\alpha}_{ii}.\]

Let $Q=|A|^{2}-\sigma|H|^{2}-bk_{1}.$  By assumption, we have $\cos^{2}\alpha \geq \frac{13\lambda-10}{3(\lambda+2)}$, then by Lemma \ref{ECO}, we have $R_{1234}\geq 0$.  Since $|\overline{\nabla}J_{\Sigma_{t}}|^{2}\geq \frac{1}{2}|H|^{2}$, we have
\begin{eqnarray}
\begin{aligned}
&(\frac{\partial}{\partial t}-\Delta)Q\\
=&(\frac{\partial}{\partial t}-\Delta)(|A|^{2}-\sigma |H|^{2})\\
\leq&-2(|\nabla A|^{2}-\sigma |\nabla H|^{2})+8(R_{1234}-R_{1212})(Q+bk_{1})+[8\sigma R_{1234}+(4-8\sigma)R_{1212}]|H|^{2}\\
&-8R_{1234}|\overline{\nabla}J_{\Sigma_{t}}|^{2}+2R_{33}(Q+bk_{1})+2(R_{44}-R_{33})(|h_{4}|^{2}-\sigma (H^{4})^{2})+4R_{34}(h_{ij}^{3}h_{ij}^{4}-\sigma H^{3}H^{4})\\
&-2R_{3434}(Q+bk_{1})+2P_{1}+2P_{2}-2\sigma P_{3}+8Kk_{1}(\lambda-1)+8Kk_{1}(\lambda-1)(Q+\sigma |H|^{2}+bk_{1})\\
\leq&-2(|\nabla A|^{2}-\sigma |\nabla H|^{2})+8(R_{1234}-R_{1212})(Q+bk_{1})+[8\sigma R_{1234}+(4-8\sigma)R_{1212}]|H|^{2}\\
&-4R_{1234}|H|^{2}+2R_{33}(Q+bk_{1})+2(R_{44}-R_{33})(|h_{4}|^{2}-\sigma (H^{4})^{2})+4R_{34}(h_{ij}^{3}h_{ij}^{4}-\sigma H^{3}H^{4})\\
&-2R_{3434}(Q+bk_{1})+2P_{1}+2P_{2}-2\sigma P_{3}+8Kk_{1}(\lambda-1)+8Kk_{1}(\lambda-1)(Q+\sigma |H|^{2}+bk_{1})\\
=&-2(|\nabla A|^{2}-\sigma |\nabla H|^{2})+CQ+[(8\sigma-4)(R_{1234}-R_{1212})+8\sigma Kk_{1}(\lambda-1)]|H|^{2}\\&+2(R_{44}-R_{33})(|h_{4}|^{2}-\sigma (H^{4})^{2})+4R_{34}(h_{ij}^{3}h_{ij}^{4}-\sigma H^{3}H^{4})+2P_{1}+2P_{2}-2\sigma P_{3}\\
&+[8(R_{1234}-R_{1212})+2R_{33}-2R_{3434}+8Kk_{1}(\lambda-1)]bk_{1}+8Kk_{1}(\lambda-1),
\end{aligned}
\end{eqnarray}
here $C$ is a function. In the following, $C$ always means a function, which might be different at places.

In Lemma \ref{L3.1}, we choose $\eta=\frac{3}{4}-\sigma$, then
\begin{equation*}
|\nabla A|^{2}\geq\sigma|\nabla H|^{2}-\frac{4\sigma-2}{3-4\sigma}|w|^{2},
\end{equation*}
where
$$|w|^{2}=R^{2}_{3212}+R^{2}_{3121}+R^{2}_{4121}+R^{2}_{4212}.$$
Assume $\alpha\in[0,\frac{\pi}{2})$ and $y\geq 0, z\geq 0$. Since $y^{2}+z^{2}=\sin^{2}\alpha$, then by Lemma \ref{ECO}, we have
\begin{eqnarray*}
\begin{aligned}
&|w|^{2}\\
\leq &\frac{k^{2}_{1}}{512}\{[(23+6(\cos^{2}\alpha+z^{2}))^{2}+(23+6(\cos^{2}\alpha+y^{2}))^{2}(\lambda-1)^{2}\\
&+12\cos\alpha[(23+6(\cos^{2}\alpha+z^{2}))y+(23+6(\cos^{2}\alpha+y^{2}))z](\lambda^{2}-1)\\
&+144\sin^{2}\alpha\cos^{2}\alpha(\lambda+1)^{2}\}\\
\leq & \frac{k^{2}_{1}}{512}[1682(\lambda-1)^{2}+261(\lambda^{2}-1)+144\sin^{2}\alpha\cos^{2}\alpha(\lambda+1)^{2},
\end{aligned}
\end{eqnarray*}
where we have used Cauchy inequality,
\begin{eqnarray*}
\begin{aligned}
&[(23+6(\cos^{2}\alpha+z^{2}))y+(23+6(\cos^{2}\alpha+y^{2}))z]^{2}\\
\leq& [(23+6(\cos^{2}\alpha+z^{2}))^{2}+(23+6(\cos^{2}\alpha+y^{2}))^{2}](y^{2}+z^{2})\\
\leq&2\cdot 29^{2}\sin^{2}\alpha.
\end{aligned}
\end{eqnarray*}
Since $1\leq \lambda\leq 1+\frac{1}{100}$, we have
\[|w|^{2}\leq \frac{k^{2}_{1}}{32}[34(\lambda-1)+9\sin^{2}\alpha\cos^{2}\alpha(\lambda+1)^{2}].\]

Since $\frac{1}{2}<\sigma<\frac{3}{4}$, we have
\begin{equation}
|\nabla A|^{2}\geq\sigma|\nabla H|^{2}-\frac{(2\sigma-1)k^{2}_{1}}{16(3-4\sigma)}[34(\lambda-1)+9\sin^{2}\alpha\cos^{2}\alpha(\lambda+1)^{2}].
\end{equation}

Hence we obtain
\begin{eqnarray}
\begin{aligned}
&(\frac{\partial}{\partial t}-\Delta)Q\\
\leq & CQ+[(8\sigma-4)(R_{1234}-R_{1212})+8\sigma Kk_{1}(\lambda-1)]|H|^{2}\\&+2(R_{44}-R_{33})(|h_{4}|^{2}-\sigma (H^{4})^{2})+4R_{34}(h_{ij}^{3}h_{ij}^{4}-\sigma H^{3}H^{4})+2P_{1}+2P_{2}-2\sigma P_{3}\\
&+[8(R_{1234}-R_{1212})+2R_{33}-2R_{3434}+8Kk_{1}(\lambda-1)]bk_{1}+8Kk_{1}(\lambda-1)\\
&+\frac{(2\sigma-1)k^{2}_{1}}{8(3-4\sigma)}[34(\lambda-1)+9\sin^{2}\alpha\cos^{2}\alpha(\lambda+1)^{2}].
\end{aligned}
\end{eqnarray}

At the point $|H|\neq 0$, we choose $\{e_{3},e_{4}\}$ for $N\Sigma$ such that $e_{3}=H/|H|$ and choose $\{e_{1},e_{2}\}$ for $T\Sigma$ such that $h_{ij}^{3}=\lambda_{i}\delta_{ij}$. Set $h_{ij}^{\alpha}=\mathring{h}_{ij}^{\alpha}+\frac{1}{2}H^{\alpha}g_{ij}$, then $\mathring{h}_{ij}^{4}=h_{ij}^{4}, \mathring{h}_{ij}^{3}=h_{ij}^{3}-\frac{1}{2}|H|g_{ij}$. Denote the norm of $(h_{ij}^{\alpha})$, $(\mathring{h}_{ij}^{\alpha})$ by $|h_{\alpha}|, |\mathring{h}_{\alpha}|$ respectively.  Then $H^{3}=|H|$ and $H^{4}=0$. From the calculation from the proof of Theorem 3.2 in [HLY], we have
\begin{eqnarray*}
\begin{aligned}
&2P_{1}+2P_{2}-2\sigma P_{3}\\
\leq &2|\mathring{h}_{3}|^{4}+2|\mathring{h}_{4}|^{4}+(2-2\sigma)|\mathring{h}_{3}|^{2}|H|^{2}-\frac{2\sigma-1}{2}|H|^{4}+8|\mathring{h}_{3}|^{2}|\mathring{h}_{4}|^{2}.
\end{aligned}
\end{eqnarray*}

Then we have
\begin{eqnarray*}
\begin{aligned}
&(\frac{\partial}{\partial t}-\Delta)Q\\
\leq& CQ+[(8\sigma-4)(R_{1234}-R_{1212})+8\sigma Kk_{1}(\lambda-1)]|H|^{2}\\
&+2(R_{44}-R_{33})|\mathring{h}_{4}|^{2}+4R_{34}\mathring{h}_{ij}^{3}\mathring{h}_{ij}^{4}\\
&+2|\mathring{h}_{3}|^{4}+2|\mathring{h}_{4}|^{4}+(2-2\sigma)|\mathring{h}_{3}|^{2}|H|^{2}-\frac{2\sigma-1}{2}|H|^{4}+8|\mathring{h}_{3}|^{2}|\mathring{h}_{4}|^{2}\\
&+[8(R_{1234}-R_{1212})+2R_{33}-2R_{3434}+8Kk_{1}(\lambda-1)]bk_{1}+8Kk_{1}(\lambda-1)\\
&+\frac{(2\sigma-1)k^{2}_{1}}{8(3-4\sigma)}[34(\lambda-1)+9\sin^{2}\alpha\cos^{2}\alpha(\lambda+1)^{2}]
\end{aligned}
\end{eqnarray*}

Since $|H|^{2}=\frac{2}{2\sigma-1}(|\mathring{h}_{3}|^{2}+|\mathring{h}_{4}|^{2}-Q-bk_{1})$, then
\begin{eqnarray*}
\begin{aligned}
&(\frac{\partial}{\partial t}-\Delta)Q\\
\leq& CQ+\frac{2}{2\sigma-1}[(8\sigma-4)(R_{1234}-R_{1212})+8\sigma Kk_{1}(\lambda-1)](|\mathring{h}_{3}|^{2}+|\mathring{h}_{4}|^{2})\\
&+2(R_{44}-R_{33})|\mathring{h}_{4}|^{2}+2|R_{34}|(|\mathring{h}_{3}|^{2}+|\mathring{h}_{4}|^{2})+2|\mathring{h}_{3}|^{4}+2|\mathring{h}_{4}|^{4}\\
&+\frac{4-4\sigma}{2\sigma-1}|\mathring{h}_{3}|^{2}(|\mathring{h}_{3}|^{2}+|\mathring{h}_{4}|^{2}-bk_{1})-\frac{2}{2\sigma-1}(|\mathring{h}_{3}|^{2}+|\mathring{h}_{4}|^{2}-bk_{1})^{2}+8|\mathring{h}_{3}|^{2}|\mathring{h}_{4}|^{2}\\
&+[8(R_{1234}-R_{1212})+2R_{33}-2R_{3434}+8Kk_{1}(\lambda-1)]bk_{1}+8Kk_{1}(\lambda-1)\\
&+\frac{(2\sigma-1)k^{2}_{1}}{8(3-4\sigma)}[34(\lambda-1)+9\sin^{2}\alpha\cos^{2}\alpha(\lambda+1)^{2}]\\
&-\frac{2}{2\sigma-1}[(8\sigma-4)(R_{1234}-R_{1212})+8\sigma Kk_{1}(\lambda-1)]bk_{1}\\
\leq& CQ+\frac{4\sigma-4}{2\sigma-1}|\mathring{h}_{4}|^{4}+\frac{12\sigma-8}{2\sigma-1}|\mathring{h}_{3}|^{2}|\mathring{h}_{4}|^{2}+\frac{4\sigma-4}{2\sigma-1}bk_{1}|\mathring{h}_{3}|^{2}\\
&+[8(R_{1234}-R_{1212})+2|R_{34}|+2|R_{44}-R_{33}|+\frac{16\sigma Kk_{1}(\lambda-1)}{2\sigma-1}+\frac{4bk_{1}}{2\sigma-1}](|\mathring{h}_{3}|^{2}+|\mathring{h}_{4}|^{2})\\
&+[2R_{33}-2R_{3434}+8Kk_{1}(\lambda-1)-\frac{16\sigma Kk_{1}}{2\sigma-1}(\lambda-1)]bk_{1}+8Kk_{1}(\lambda-1)\\
&+\frac{(2\sigma-1)k^{2}_{1}}{8(3-4\sigma)}[34(\lambda-1)+9\sin^{2}\alpha\cos^{2}\alpha(\lambda+1)^{2}]-\frac{2b^{2}k^{2}_{1}}{2\sigma-1}\\
=& CQ+\frac{4\sigma-4}{2\sigma-1}(|\mathring{h}_{4}|^{2}-\frac{bk_{1}}{2})^{2}+\frac{12\sigma-8}{2\sigma-1}|\mathring{h}_{3}|^{2}|\mathring{h}_{4}|^{2}\\
&+[8(R_{1234}-R_{1212})+2|R_{34}|+2|R_{44}-R_{33}|+\frac{16\sigma Kk_{1}(\lambda-1)}{2\sigma-1}+\frac{4\sigma bk_{1}}{2\sigma-1}](|\mathring{h}_{3}|^{2}+|\mathring{h}_{4}|^{2})\\
&+[2R_{33}-2R_{3434}+8Kk_{1}(\lambda-1)-\frac{16\sigma Kk_{1}}{2\sigma-1}(\lambda-1)]bk_{1}+8Kk_{1}(\lambda-1)\\
&+\frac{(2\sigma-1)k^{2}_{1}}{8(3-4\sigma)}[34(\lambda-1)+9\sin^{2}\alpha\cos^{2}\alpha(\lambda+1)^{2}]-\frac{\sigma+1}{2\sigma-1}b^{2}k^{2}_{1}
\end{aligned}
\end{eqnarray*}

Set
\[C_{1}=8(R_{1234}-R_{1212})+2|R_{34}|+2|R_{44}-R_{33}|+\frac{16\sigma Kk_{1}(\lambda-1)}{2\sigma-1}+\frac{4\sigma bk_{1}}{2\sigma-1},\]
and
\begin{eqnarray*}
\begin{aligned}
C_{2}=&[2R_{33}-2R_{3434}+8Kk_{1}(\lambda-1)-\frac{16\sigma Kk_{1}}{2\sigma-1}(\lambda-1)]bk_{1}+8Kk_{1}(\lambda-1)\\
&+\frac{(2\sigma-1)k^{2}_{1}}{8(3-4\sigma)}[34(\lambda-1)+9\sin^{2}\alpha\cos^{2}\alpha(\lambda+1)^{2}]-\frac{\sigma+1}{2\sigma-1}b^{2}k^{2}_{1}.
\end{aligned}
\end{eqnarray*}

Then we have
\begin{eqnarray}
\begin{aligned}
(\frac{\partial}{\partial t}-\Delta)Q\leq& CQ+\frac{4\sigma-4}{2\sigma-1}(|\mathring{h}_{4}|^{2}-\frac{bk_{1}}{2})^{2}+\frac{12\sigma-8}{2\sigma-1}|\mathring{h}_{3}|^{2}|\mathring{h}_{4}|^{2}\\
&+C_{1}(|\mathring{h}_{3}|^{2}+|\mathring{h}_{4}|^{2})+C_{2}.
\end{aligned}
\end{eqnarray}

By Lemma \ref{ECO}, we can estimate $C_{1}$ as following
\begin{eqnarray*}
\begin{aligned}
C_{1}\leq &\frac{2}{3}[(10+6\cos^{2}\alpha)\lambda-(10+3\sin^{2}\alpha)]k_{1}-2[(3+3\cos^{2}\alpha)-2\lambda]k_{1}\\
&+\frac{(29-6\cos^{2}\alpha)(\lambda-1)}{8}k_{1}+9(\lambda-1)k_{1}+\frac{16\sigma Kk_{1}(\lambda-1)}{2\sigma-1}+\frac{4\sigma bk_{1}}{2\sigma-1}\\
=&[\frac{559+78\cos^{2}\alpha}{24}\lambda-\frac{655+78\cos^{2}\alpha}{24}+\frac{16K(\lambda-1)}{2\sigma-1}+\frac{4\sigma b}{2\sigma-1}]k_{1}\\
\leq& [32(\lambda-1)-4\lambda+\frac{16\sigma K(\lambda-1)}{2\sigma-1}+\frac{4\sigma b}{2\sigma-1}]k_{1}
\end{aligned}
\end{eqnarray*}

Then by the condition $b= \frac{2\sigma-1}{\sigma}(8-7\lambda)-4K(\lambda-1)$, we know $C_{1}\leq 0$.

By Lemma \ref{ECO} and since $K=\min\{2,2k_{1}\}$, we can estimate $C_{2}$ as following
\begin{eqnarray*}
\begin{aligned}
C_{2}\leq& [7\lambda-\frac{9+3\cos^{2}\alpha}{2}+16(\lambda-1)-\frac{16K\sigma (\lambda-1)}{2\sigma-1}]bk^{2}_{1}\\
&+16k^{2}_{1}(\lambda-1)+\frac{(2\sigma-1)k^{2}_{1}}{8(3-4\sigma)}[34(\lambda-1)+9\sin^{2}\alpha\cos^{2}\alpha(\lambda+1)^{2}]-\frac{\sigma+1}{2\sigma-1}b^{2}k^{2}_{1}\\
=&\{[23\lambda-\frac{41+3\cos^{2}\alpha}{2}-\frac{16K\sigma (\lambda-1)}{2\sigma-1}]b-\frac{\sigma+1}{2\sigma-1}b^{2}\\
&+16(\lambda-1)+\frac{(2\sigma-1)}{8(3-4\sigma)}[34(\lambda-1)+9\sin^{2}\alpha\cos^{2}\alpha(\lambda+1)^{2}]\}k^{2}_{1}
\end{aligned}
\end{eqnarray*}

Set $t=\cos^{2}\alpha$. Then
\begin{eqnarray*}
\begin{aligned}
C_{2}\leq& \frac{(2\sigma-1)k_{1}^{2}}{8(3-4\sigma)}\{-9(\lambda+1)^{2}t^{2}+[9(\lambda+1)^{2}-\frac{12(3-4\sigma)}{2\sigma-1}b]t\\
&+\frac{350-444\sigma}{2\sigma-1}(\lambda-1)+\frac{8(3-4\sigma)}{2\sigma-1}(23\lambda-\frac{41}{2})b-\frac{8(3-4\sigma)(\sigma+1)}{(2\sigma-1)^{2}}b^{2}\}
\end{aligned}
\end{eqnarray*}
and
\begin{equation}
f(t)=-a_{1}t^{2}+a_{2}t+a_{3},
\end{equation}
where
\[a_{1}=9(\lambda+1)^{2},\]
\[a_{2}=9(\lambda+1)^{2}-\frac{12(3-4\sigma)}{2\sigma-1}b,\]
\[a_{3}=\frac{350-444\sigma}{2\sigma-1}(\lambda-1)+\frac{8(3-4\sigma)}{2\sigma-1}(23\lambda-\frac{41}{2})b-\frac{8(3-4\sigma)(\sigma+1)}{(2\sigma-1)^{2}}b^{2}.\]
\[t_{0}=\frac{a_{2}+\sqrt{a_{2}^{2}+4a_{1}a_{3}}}{2a_{1}}.\]
Then $C_{2}\leq \frac{(2\sigma-1)k_{1}^{2}}{8(3-4\sigma)}f(t)$ and $f(t_{0})=0$.

We first to check that $f(1)<0$, i.e., $-a_{1}+a_{2}+a_{3}<0,$ then implies that $t_{0}<1$. That is
\[\frac{350-444\sigma}{2\sigma-1}(\lambda-1)+\frac{8(3-4\sigma)}{2\sigma-1}(23\lambda-22)b-\frac{8(3-4\sigma)(\sigma+1)}{(2\sigma-1)^{2}}b^{2}<0.\]

Set $x=\lambda -1$. Then $b=\frac{2\sigma-1}{\sigma}(1-7x)-4Kx,$ we have
\begin{eqnarray*}
\begin{aligned}
f(1)=&\frac{350-444\sigma}{2\sigma-1}x+\frac{8(3-4\sigma)}{2\sigma-1}(23x+1)(\frac{2\sigma-1}{\sigma}(1-7x)-4Kx)\\
&-\frac{8(3-4\sigma)(\sigma+1)}{(2\sigma-1)^{2}}(\frac{2\sigma-1}{\sigma}(1-7x)-4Kx)^{2}\\
=&\frac{350-444\sigma}{2\sigma-1}x+\frac{8(3-4\sigma)}{\sigma}(23x+1)(1-7x)-\frac{32(3-4\sigma)Kx}{2\sigma-1}\\
&-\frac{8(3-4\sigma)(\sigma+1)}{\sigma^{2}}(1-(7+\frac{4\sigma K}{2\sigma-1})x)^{2}\\
\leq&\frac{350-444\sigma}{2\sigma-1}x+\frac{8(3-4\sigma)}{\sigma}(23x+1)(1-7x)-\frac{32(3-4\sigma)}{2\sigma-1}Kx\\
&-\frac{8(3-4\sigma)(\sigma+1)}{\sigma^{2}}+\frac{16(3-4\sigma)(\sigma+1)}{\sigma^{2}}(7+\frac{4\sigma K}{2\sigma-1})x\\
\leq&\frac{350-444\sigma}{2\sigma-1}x+\frac{8(3-4\sigma)}{\sigma}(1+16x)-\frac{8(3-4\sigma)(\sigma+1)}{\sigma^{2}}\\
&+\frac{112(3-4\sigma)(\sigma+1)x}{\sigma^{2}}+\frac{64(3-4\sigma)Kx}{\sigma(2\sigma-1)}\\
\leq&\frac{350-444\sigma}{2\sigma-1}x+\frac{8(3-4\sigma)}{\sigma}(1+16x)-\frac{8(3-4\sigma)(\sigma+1)}{\sigma^{2}}\\
&+\frac{112(3-4\sigma)(\sigma+1)x}{\sigma^{2}}+\frac{128(3-4\sigma)x}{\sigma(2\sigma-1)}\\
\leq&\frac{8(3-4\sigma)}{\sigma^{2}}(-1+(30\sigma+14)x)+(222-956\sigma+\frac{384}{\sigma})\frac{x}{2\sigma-1}\\
\leq&\frac{2}{\sigma^{2}}[4(3-4\sigma)(-1+(30\sigma+14)x)+(-478\sigma^{3}+111\sigma^{2}+192\sigma)\frac{x}{2\sigma-1}]\\
\leq &\frac{2}{\sigma^{2}}(\frac{8}{3}(-1+34x)+128\frac{x}{2\sigma-1})\\
= &\frac{16}{3\sigma^{2}}(-1+34x+48\frac{x}{2\sigma-1}).
\end{aligned}
\end{eqnarray*}
Since $0\leq x\leq \frac{1}{200}$ and $\frac{1}{2}+\frac{24x}{1-34x}\leq \frac{2}{3}$, we obtain that $-1+34x+48\frac{x}{2\sigma-1}<0$. So $f(1)<0$.
Then we have $f(t)\leq 0$ for $t_{0}\leq t\leq 1$. Hence $C_{2}\leq 0$ for $\cos\alpha\geq t_{0}.$

At the point $|H|=0$, it is easy to obtain that (also see the inequality (3.12) in [HLY])
\begin{equation}
2P_{1}+2P_{2}\leq 3|A|^{4},
\end{equation}
and $P_{3}=0$. Thus, $|A|^{2}=Q+bk_{1}$, by (38) we have
\begin{eqnarray*}
\begin{aligned}
&(\frac{\partial}{\partial t}-\Delta)Q\\
\leq&-CQ+2(R_{44}-R_{33})|h_{4}|^{2}+4R_{34}h_{ij}^{3}h_{ij}^{4}+3|A|^{4}\\
&+[8(R_{1234}-R_{1212})+2R_{33}-2R_{3434}+8Kk_{1}(\lambda-1)]bk_{1}+8Kk_{1}(\lambda-1)\\
&+\frac{(2\sigma-1)k_{1}^{2}}{8(3-4\sigma)}[34(\lambda-1)+9\sin^{2}\alpha\cos^{2}\alpha(\lambda+1)^{2}]\\
\leq&-CQ+2|R_{44}-R_{33}||A|^{2}+2|R_{34}||A|^{2}+3(Q+bk_{1})^{2}\\
&+[8(R_{1234}-R_{1212})+2R_{33}-2R_{3434}+8Kk_{1}(\lambda-1)]bk_{1}+8Kk_{1}(\lambda-1)\\
&+\frac{(2\sigma-1)k_{1}^{2}}{8(3-4\sigma)}[34(\lambda-1)+9\sin^{2}\alpha\cos^{2}\alpha(\lambda+1)^{2}]\\
=&-CQ+2|R_{44}-R_{33}|(Q+bk_{1})+2|R_{34}|(Q+bk_{1})\\
&+[8(R_{1234}-R_{1212})+2R_{33}-2R_{3434}+8Kk_{1}(\lambda-1)]bk_{1}+8Kk_{1}(\lambda-1)\\
&+\frac{(2\sigma-1)k_{1}^{2}}{8(3-4\sigma)}[34(\lambda-1)+9\sin^{2}\alpha\cos^{2}\alpha(\lambda+1)^{2}]+3b^{2}k_{1}^{2}\\
=&-CQ+\frac{(2\sigma-1)k_{1}^{2}}{8(3-4\sigma)}[34(\lambda-1)+9\sin^{2}\alpha\cos^{2}\alpha(\lambda+1)^{2}]+3b^{2}k_{1}^{2}+8Kk_{1}(\lambda-1)\\
&+[2|R_{44}-R_{33}|+2|R_{34}|+8(R_{1234}-R_{1212})+2R_{33}-2R_{3434}+8Kk_{1}(\lambda-1)]bk_{1}
\end{aligned}
\end{eqnarray*}

Set
\begin{eqnarray*}
\begin{aligned}
\tilde{C}_{2}=&\frac{(2\sigma-1)k_{1}^{2}}{8(3-4\sigma)}[34(\lambda-1)+9\sin^{2}\alpha\cos^{2}\alpha(\lambda+1)^{2}]+3b^{2}k_{1}^{2}+8Kk_{1}(\lambda-1)\\
&+[2|R_{44}-R_{33}|+2|R_{34}|+8(R_{1234}-R_{1212})+2R_{33}-2R_{3434}+8Kk_{1}(\lambda-1)]bk_{1}.
\end{aligned}
\end{eqnarray*}
Then it is easy to get
\begin{equation*}
\tilde{C}_{2}=C_{1}bk_{1}+C_{2}+\frac{3\sigma-2}{2\sigma-1}b^{2}k_{1}^{2}.
\end{equation*}
Since $\frac{1}{2}<\sigma\leq \frac{2}{3}$, $C_{1}\leq 0$ and $C_{2}\leq 0$, we have $\tilde{C}_{2}\leq 0$.

Therefore, from the above argument, we have
\begin{equation*}
(\frac{\partial}{\partial t}-\Delta)Q\leq CQ.
\end{equation*}
for some function $C$. Applying the maximum principle for the above parabolic equation, we know that
\begin{equation*}
Q\leq 0
\end{equation*}
along the flow, if it is true on initial surface.
\end{proof}
\section{Long time existence and convergence}
In this section we prove the long time existence and convergence of the symplectic mean curvature flow under the assumption of Theorem \ref{Thm1.3}, using the same argument in [HLY]. For convenience of readers, we provide the detailed proof here.
\begin{theorem}
Under the assumption of Theorem \ref{Thm1.3}, the symplectic mean curvature flow exists for long time.
\end{theorem}
\begin{proof}
Suppose $f$ is a positive increasing function which will be determined later. Now we compute the evolution equation of $|H|^{2}f(\frac{1}{\cos\alpha})$.
\begin{eqnarray*}
\begin{aligned}
&(\frac{\partial}{\partial t}-\Delta)(|H|^{2}f(\frac{1}{\cos\alpha}))\\
=&(\frac{\partial}{\partial t}-\Delta)|H|^{2}f(\frac{1}{\cos\alpha})+|H|^{2}(\frac{\partial}{\partial t}-\Delta)f(\frac{1}{\cos\alpha})
-2\nabla |H|^{2}\cdot\nabla f(\frac{1}{\cos\alpha}).
\end{aligned}
\end{eqnarray*}
Under the assumptions of the theorem, we have
\begin{eqnarray*}
\begin{aligned}
(\frac{\partial}{\partial t}-\Delta)\cos\alpha&=|\overline\nabla J_{\Sigma_{t}}|^{2}\cos\alpha+Ric(Je_{1}, e_{2})\sin^{2}\alpha\\
&\geq |\overline\nabla J_{\Sigma_{t}}|^{2}\cos\alpha.
\end{aligned}
\end{eqnarray*}
By the evolution equation of $|H|^{2}$, we have
\begin{eqnarray*}
\begin{aligned}
&(\frac{\partial}{\partial t}-\Delta)|H|^{2}\\
=&-2|\nabla H|^{2}+2R_{33}(H^{3})^{2}+2R_{44}(H^{4})^{2}+4R_{34}H^{3}H^{4}-2R_{3434}|H|^{2}+2P_{3}\\
\leq &-2|\nabla H|^{2}+(6k_{2}-3k_{1})|H|^{2}-\frac{1}{2}[(3+3\cos^{2}\alpha)k_{1}-2k_{2}]|H|^{2}\\
&+\frac{1}{8}(23+6\sin^{2}\alpha)(k_{2}-k_{1})|H|^{2}+2|H|^{2}|A|^{2}\\
=&-2|\nabla H|^{2}+[\frac{5}{8}(17\lambda-13)-\frac{3}{4}(\lambda+1)\cos^{2}\alpha]k_{1}|H|^{2}+2|H|^{2}|A|^{2}.
\end{aligned}
\end{eqnarray*}
We have used $P_{3}\leq |H|^{2}|A|^{2}$ in the above equation.

Denote $C=\frac{5}{8}(17\lambda-13)-\frac{3}{4}(\lambda+1)\delta$. Since $\cos\alpha\geq \sqrt{\delta}>0$ and the pinching condition, we have
\begin{eqnarray*}
\begin{aligned}
&(\frac{\partial}{\partial t}-\Delta)|H|^{2}\\
\leq &-2|\nabla H|^{2}+Ck_{1}|H|^{2}+2|H|^{2}(\sigma |H|^{2}+bk_{1})\\
=&-2|\nabla H|^{2}+(C+2b)k_{1}|H|^{2}+2\sigma |H|^{4}.
\end{aligned}
\end{eqnarray*}
Putting the above inequality into the evolution of $|H|^{2}f(\frac{1}{\cos\alpha})$, we get that
\begin{eqnarray}
\begin{aligned}
&(\frac{\partial}{\partial t}-\Delta)(|H|^{2}f(\frac{1}{\cos\alpha}))\\
\leq & f(\frac{1}{\cos\alpha})(-2|\nabla H|^{2}+(C+2b)k_{1}|H|^{2}+2\sigma |H|^{4})\\
&-|H|^{2}(f'\frac{|\overline{\nabla}J_{\Sigma_{t}}|^{2}}{\cos\alpha}+2f'\frac{|\nabla \cos\alpha|^{2}}{\cos^{3}\alpha}+f''\frac{|\nabla\cos\alpha|^{2}}{\cos^{4}\alpha})\\
&-2\frac{\nabla(f|H|^{2})-|H|^{2}\nabla f}{f}\cdot \nabla f(\frac{1}{\cos\alpha})\\
=&(C+2b)k_{1}f|H|^{2}+|H|^{2}f(\frac{1}{\cos\alpha})(-2\frac{|\nabla H|^{2}}{|H|^{2}}+2\sigma |H|^{2}-\frac{f'}{f}\frac{|\overline{\nabla}J_{\Sigma_{t}}|^{2}}{\cos\alpha})\\
&+|H|^{2}(-f''+2\frac{(f')^{2}}{f}-2f'\cos\alpha)\frac{|\nabla \cos\alpha|^{2}}{\cos^{4}\alpha}\\
&-2|H|^{2}\frac{\nabla(f|H|^{2})}{f|H|^{2}}\cdot \nabla f(\frac{1}{\cos\alpha})\end{aligned}
\end{eqnarray}
Set $\phi=f|H|^{2}$. At the point where $\phi\neq 0$, it is easy to see that
\begin{equation*}
\nabla \phi=f\nabla|H|^{2}+|H|^{2}\nabla f=f\nabla|H|^{2}-|H|^{2}f'\frac{\nabla\cos\alpha}{\cos^{2}\alpha},
\end{equation*}
i.e.,
\begin{equation}
\frac{\nabla\cos\alpha}{\cos^{2}\alpha}=\frac{f}{f'}(\frac{\nabla|H|^{2}}{|H|^{2}}-\frac{\nabla\phi}{\phi}).
\end{equation}
Plugging (52) into (51), we obtain
\begin{eqnarray}
\begin{aligned}
&(\frac{\partial}{\partial t}-\Delta)\phi\\
\leq &(C+2b)k_{1}\phi+\phi(-2\frac{|\nabla H|^{2}}{|H|^{2}}+2\sigma|H|^{2}-\frac{f'}{f}\frac{\overline{\nabla}J_{\Sigma_{t}}|^{2}}{\cos\alpha})\\
&+\frac{\phi f}{(f')^{2}}(-f''+2\frac{(f')^{2}}{f}-2f'\cos\alpha)(\frac{|\nabla|H|^{2}|^{2}}{|H|^{4}}-2\frac{\nabla|H|^{2}}{|H|^{2}}\cdot \frac{\nabla \phi}{\phi}+\frac{|\nabla\phi|^{2}}{\phi^{2}})\\
&+2|H|^{2}f'\frac{\nabla\phi}{\phi}\cdot\frac{\nabla\cos\alpha}{\cos^{2}\alpha}\\
\leq&(C+2b)k_{1}k_{1}\phi+\phi(-\frac{f'}{2f}\frac{|H|^{2}}{\cos\alpha}+2\sigma|H|^{2})\\
&+\phi(-2\frac{|\nabla H|^{2}}{|H|^{2}}-4\frac{ff''}{(f')^{2}}\frac{|\nabla|H||^{2}}{|H|^{2}}+8\frac{|\nabla|H||^{2}}{|H|^{2}}-8\frac{f|\nabla|H||\cos\alpha}{f'|H|^{2}})\\
&+2|H|^{2}f'\frac{\nabla\phi}{\phi}\cdot\frac{\nabla\cos\alpha}{\cos^{2}\alpha}+\phi(-\frac{ff''}{(f')^{2}}-2\frac{f}{f'}\cos\alpha+2)(\frac{|\nabla\phi|^{2}}{\phi^{2}}-2\frac{\nabla|H|^{2}}{|H|^{2}}\cdot\frac{\nabla\phi}{\phi})\\
=&(C+2b)k_{1}\phi+\phi(-\frac{f'}{2f\cos\alpha}+2\sigma)|H|^{2}+\phi(-4\frac{ff''}{(f')^{2}}-8\frac{f\cos\alpha}{f'}+6)\frac{|\nabla|H||^{2}}{|H|^{2}}\\
&+\phi(-\frac{ff''}{(f')^{2}}-2\frac{f}{f'}\cos\alpha+2)(\frac{|\nabla\phi|^{2}}{\phi^{2}}-2\frac{\nabla|H|^{2}}{|H|^{2}}\cdot\frac{\nabla\phi}{\phi})+2|H|^{2}f'\frac{\nabla\phi}{\phi}\cdot\frac{\nabla\cos\alpha}{\cos^{2}\alpha}
\end{aligned}
\end{eqnarray}
Set $\frac{f}{f'}=g$, we choose $g$ such that for $x\in[1,\frac{1}{\sqrt{\delta}}]$
\[\begin{cases}
x/g\geq 4\sigma,\\
-4g'+8g/x-2=0.
\end{cases}\]
We choose $c(x)=\frac{1}{2}-ax$ by solving the last equation, where $a$ will be defined later. It reduces to solve the inequality
\[0<\frac{1}{2}-ax\leq \frac{1}{4\sigma},\quad x\in[1,\frac{1}{\sqrt{\delta}}],\]
i.e.,
\[(\frac{1}{2}-\frac{1}{4\sigma})\frac{1}{x}\leq a<\frac{1}{2x},\quad x\in[1,\frac{1}{\sqrt{\delta}}].\]
Note that
\[\cos^{2}\alpha\geq \frac{13\lambda-10}{3(\lambda+2)}\geq \frac{1}{3}>(1-\frac{1}{2\sigma})^{2}\]
for any $\sigma \in (\frac{1}{2}, \frac{2}{3}]$.

Hence we can choose $a=\frac{1}{2}-\frac{1}{4\sigma}$, then
\[g=\frac{x}{2}-(\frac{1}{2}-\frac{1}{4\sigma})x^{2},\]
and
\[f(x)=\frac{(1-2a)^{2}x^{2}}{(1-2ax)^{2}}=\frac{x^{2}}{(2\sigma-(2\sigma-1)x)^{2}}, \quad x\in [1,\frac{1}{\sqrt{\delta}}].\]
Then for any $x\in[1,\frac{1}{\sqrt{\delta}}]$,
\[1\leq f(x)\leq \frac{1}{(2\sigma \sqrt{\delta}-(2\sigma-1))^{2}}.\]
By (44), we have
\begin{eqnarray}
\begin{aligned}
&(\frac{\partial}{\partial t}-\Delta)\phi\\
\leq&(C+2b)k_{1}\phi+\phi(-\frac{ff''}{(f')^{2}}-2\frac{f}{f'}\cos\alpha+2)(\frac{|\nabla\phi|^{2}}{\phi^{2}}-2\frac{\nabla|H|^{2}}{|H|^{2}}\cdot\frac{\nabla\phi}{\phi})\\
&+2|H|^{2}f'\frac{\nabla\phi}{\phi}\cdot\frac{\nabla\cos\alpha}{\cos^{2}\alpha}
\end{aligned}
\end{eqnarray}
This implies that
\[\frac{\partial}{\partial t}\phi_{max}(t)\leq (C+2b)k_{1}\phi_{max}(t),\]
where $\phi_{max}(t)$ mean the maximum of $\phi$ on $\Sigma_{t}$. Hence
\[|H|^{2}(t)\leq \phi_{max}(t)\leq e^{(C+2b)k_{1}t}|H|^{2}(0)f(\frac{1}{\cos\alpha})(0).\]
We have
\[|H|^{2}(t)\leq C_{0}e^{C_{1}t},\]
where $C_{0}$ depends only on $\max_{\Sigma_{0}}|H|^{2}, \lambda$ and $\sigma$. Pinching inequality implies that
\[|A|^{2}\leq C_{2}e^{C_{1}t}+bk_{1}.\]
We finish the proof of the theorem.
\end{proof}
\begin{theorem}
Under the assumption of Theorem \ref{Thm1.3}, the symplectic mean curvature flow converges to a holomorphic curve at infinity.
\end{theorem}
In fact, this theorem can follows from Corollary 4.2.1. But here we also provide the argument from [HLY].
\begin{proof}
Since $\cos^{2}\alpha\geq\frac{1}{3}$ and $\lambda\leq 1+\frac{1}{200}$ and (32), we know $Ric(Je_{1},e_{2})>\frac{1}{3}k_{1}$. By the evolution equation of $\cos\alpha$
\[(\frac{\partial}{\partial t}-\Delta)\cos\alpha=|\overline{\nabla}J_{\Sigma_{t}}|^{2}\cos\alpha+Ric(Je_{1},e_{2})\sin^{2}\alpha,\]
we have
\[(\frac{\partial}{\partial t}-\Delta)\cos\alpha\geq \frac{1}{3}k_{1}\sin^{2}\alpha.\]
Rewrite the above inequality as
\[(\frac{\partial}{\partial t}-\Delta)\sin^{2}(\alpha/2)\leq -\frac{4}{9}k_{1}\sin^{2}(\alpha/2).\]
Applying the maximum principle, we get that $\sin^{2}(\alpha/2)\leq e^{-\frac{4}{9}k_{1}t}.$ By Theorem 6.1 we know that the symplectic mean curvature flow exists for long time. Thus for any $\epsilon>0$, there exists $T>0$ such that as $t>T$, we have
\begin{eqnarray}
\begin{aligned}
\cos\alpha&\geq 1-\epsilon,\\
\sin\alpha&\leq 2\epsilon,\\
|\nabla \cos\alpha|^{2}&\leq \sin^{2}\alpha|\overline{\nabla}J_{\Sigma_{t}}|^{2}\leq 2\epsilon |\overline{\nabla}J_{\Sigma_{t}}|^{2}\leq 4\epsilon |A|^{2}.
\end{aligned}
\end{eqnarray}
Therefore by pinching inequality,
\begin{eqnarray}
\begin{aligned}
(\frac{\partial}{\partial t}-\Delta)\cos\alpha\geq \frac{1}{2}|H|^{2}\cos\alpha&\geq(\frac{1}{2\sigma}|A|^{2}-\frac{bk_{1}}{2\sigma})\cos\alpha\\
&\geq \frac{1}{2\sigma}(1-\epsilon)|A|^{2}-\frac{bk_{1}}{2\sigma}.
\end{aligned}
\end{eqnarray}
From the evolution equation of $|A|^{2}$, we have
\[(\frac{\partial}{\partial t}-\Delta)|A|^{2}\leq -2|\nabla A|^{2}+C_{1}|A|^{4}+C_{2}|A|^{2}+C_{3},\]
where $C_{1},C_{2},C_{3}$ are positive constants that depend only on $k_{1}, \lambda, K$.

Let $p>1$ be a constant to be fixed later. For simplicity, we set $u=\cos\alpha$. Now we consider the function $\frac{|A|^{2}}{e^{pu}}.$
\begin{eqnarray*}
\begin{aligned}
&(\frac{\partial}{\partial t}-\Delta)\frac{|A|^{2}}{e^{pu}}\\
=&2\nabla(\frac{|A|^{2}}{e^{pu}})\cdot \frac{\nabla e^{pu}}{e^{pu}}+\frac{1}{e^{2pu}}[e^{pu}(\frac{\partial}{\partial t}-\Delta)|A|^{2}-|A|^{2}(\frac{\partial}{\partial t}-\Delta)e^{pu}]\\
\leq&2p\nabla(\frac{|A|^{2}}{e^{pu}})\cdot\nabla u+\frac{1}{e^{2pu}}\{e^{pu}(C_{1}|A|^{4}+C_{2}|A|^{2}+C_{3})-p|A|^{2}e^{pu}[\frac{1}{2\sigma}(1-\epsilon)|A|^{2}-\frac{bk_{1}}{2\sigma}-p|\nabla u|^{2}]\}.
\end{aligned}
\end{eqnarray*}

Using (52) we obtain that,
\[(\frac{\partial}{\partial t}-\Delta)\frac{|A|^{2}}{e^{pu}}\leq 2p\nabla(\frac{|A|^{2}}{e^{pu}})\cdot \nabla u+\frac{1}{e^{pu}}[(C_{1}-\frac{1}{2\sigma}p(1-\epsilon)+4p^{2}\epsilon)|A|^{4}+C_{4}|A|^{2}+C_{3}]\]
Set $p^{2}=\frac{1}{\epsilon}$, then
\[C_{1}-\frac{1}{2\sigma}p(1-\epsilon)+4p^{2}\epsilon=C_{1}-\frac{1}{2\sigma}\epsilon^{-\frac{1}{2}}+\frac{1}{2\sigma}\epsilon^{\frac{1}{2}}+4.\]
As $t$ is sufficiently large, i.e., $\epsilon$ is sufficiently close to $0$, we have
\[C_{1}-\frac{1}{2\sigma}\epsilon^{-\frac{1}{2}}+\frac{1}{2\sigma}\epsilon^{\frac{1}{2}}+4\leq -1.\]
Hence
\[(\frac{\partial}{\partial t}-\Delta)\frac{|A|^{2}}{e^{pu}}\leq 2p\nabla(\frac{|A|^{2}}{e^{pu}})\cdot \nabla u-\frac{|A|^{4}}{e^{2pu}}+\frac{C_{4}|A|^{2}}{e^{pu}}+\frac{C_{3}}{e^{pu}}.\]
Applying the maximum principle, we conclude that $\frac{|A|^{2}}{e^{pu}}$ is uniformly bounded, thus $|A|^{2}$ is also uniformly bounded. Thus $F(\cdot,t)$ converges to $F_{\infty}$ in $C^{2}$ as $t\rightarrow \infty$. Since $\sin^{2}(\frac{\alpha}{2})\leq e^{-\frac{4}{9}k_{1}t}$, we have $\cos\alpha \equiv 1$ at infinity. Thus the limiting surface $F_{\infty}$ is a holomorphic curve.
\end{proof}
\section*{Acknowledgments}
The author is partially supported by NSFC No. 11301017, Research Fund for the Doctoral Program of Higher Education of China and the Fundamental Research Funds for the Central Universities. He thanks Professor Jiayu Li, Xiaoli Han, Jun Sun and Liuqing Yang for many useful discussions.

\end{document}